%% file: article.tex
\documentclass{amsart}

\usepackage[utf8]{inputenc}
\usepackage[OT1]{fontenc}

\usepackage{latexsym}
\usepackage{amsthm}
\usepackage{amsmath}
\usepackage{amsfonts}
\usepackage{amssymb}
\usepackage[all, 2cell]{xypic}

\usepackage{version}
\usepackage[mathcal]{euscript}
\usepackage{mathrsfs}

\usepackage[hidelinks]{hyperref}

\usepackage{ifpdf}
\usepackage{enumitem}
\setlist[enumerate]{label=\rm{(\roman*)}}

\SelectTips{cm}{}

\theoremstyle{plain}
\newtheorem{thm}{Theorem}[subsection]
\newtheorem{thmp}{Theorem}[section]
\newtheorem*{thm*}{Theorem}
\newtheorem{lemma}[thm]{Lemma}
\newtheorem{prop}[thm]{Proposition}
\newtheorem{propp}[thmp]{Proposition}
\newtheorem{coro}[thm]{Corollary}
\newtheorem{corop}[thmp]{Corollary}

\theoremstyle{definition}
\newtheorem{defi}[thm]{Definition}
\newtheorem{defip}[thmp]{Definition}
\newtheorem{paragr}[thm]{}
\newtheorem{paragrp}[thmp]{}

\newtheorem*{notation}{Notation and terminology}
\newtheorem*{organization}{Organization of the paper}

\theoremstyle{remark}
\newtheorem{rem}[thm]{Remark}
\newtheorem{remp}[thmp]{Remark}
\newtheorem{exam}[thm]{Example}

\newtheoremstyle{dotless}{}{}{}{}{}{}{\parindent}{}
\theoremstyle{dotless}
\newtheorem*{paragr*}{}

\usepackage{etoolbox}

\patchcmd{\section}{\scshape}{\bfseries}{}{}
\makeatletter
\renewcommand{\@secnumfont}{\bfseries}
\makeatother

\include{macros}

\title[On autoequivalences of the $(\infty, 1)$-category of
$\infty$-operads]{On autoequivalences of the $(\infty, 1)$-category\\ of
$\infty$-operads}

\author[D. Ara]{Dimitri Ara}
\address{Dimitri Ara, Radboud Universiteit Nijmegen, Institute for
Mathematics, Astrophysics and Particle Physics, Heyendaalseweg 135, 6525
AJ Nijmegen, The Netherlands}
\email{d.ara@math.ru.nl}
\urladdr{http://www.math.ru.nl/~dara/}

\author[M. Groth]{Moritz Groth}
\address{Moritz Groth, Radboud Universiteit Nijmegen, Institute for
Mathematics, Astrophysics and Particle Physics, Heyendaalseweg 135, 6525
AJ Nijmegen, The Netherlands}
\email{m.groth@math.ru.nl}
\urladdr{http://www.math.ru.nl/~mgroth/}

\author[J.\,J. Gutiérrez]{Javier J.~Gutiérrez}
\address{Javier J.~Gutiérrez, Radboud Universiteit Nijmegen, Institute for
Mathematics, Astrophysics and Particle Physics, Heyendaalseweg 135, 6525 AJ
Nijmegen, The Netherlands}
\email{j.gutierrez@math.ru.nl}
\urladdr{http://www.math.ru.nl/~gutierrez/}

\date{}

\addtocontents{toc}{\protect\setcounter{tocdepth}{1}}

\begin{document}

\begin{abstract}
We study the $(\infty, 1)$-category of autoequivalences of \oo-operads. Using
techniques introduced by Toën, Lurie, and Barwick and Schommer\nbd-Pries, we
prove that this $(\infty, 1)$-category is a contractible \oo-groupoid.  Our
calculation is based on the model of complete dendroidal Segal spaces
introduced by Cisinski and Moerdijk. Similarly, we prove that the $(\infty,
1)$-category of autoequivalences of non-symmetric \oo-operads is the discrete
monoidal category associated to~$\Z/2\Z$. We also include a computation
of the $(\infty,1)$\nbd-category of autoequivalences of $(\infty, n)$-categories
based on Rezk's $\Theta_n$-spaces.
\end{abstract}

\maketitle

\tableofcontents

\section{Introduction}
\label{sec:intro}

Higher category theory and higher operad theory can be formalized by means
of a plethora of different explicit approaches, all of them having their
merits and drawbacks. These theories have applications in fields as
diverse as algebraic topology, (derived) algebraic geometry, representation
theory and homological algebra; see for instance the foundational work of
Toën--Vezzosi \cite{HAG2DAG, HAGI, HAGII} and Lurie \cite{LurieHTT,
LurieDAG, LurieHA}. Having a specific problem at hand, we are hence able to
choose an approach accordingly, and it is thus important for practical and
theoretical purposes to know how to compare these different formulations.

In the case of $(\infty,1)$-categories, the state of the art is very
satisfactory. By now, there are many different approaches to the theory of
$(\infty,1)$-categories, including quasi-categories \cite{JoyalQCatKan},
simplicial categories \cite{BergnerSimpCat}, Segal categories
\cite{SimpsonHirsch} and complete Segal spaces \cite{RezkSegSp}. Each of
these theories is organized in a Quillen model category and these are
related by a web of Quillen equivalences; see \cite{BergnerSurvey} for a
survey on these Quillen equivalences.

In \cite{ToenAxm}, To\"en took this one step further and, based on earlier
work of \hbox{Simpson} \cite{Simpson}, offered an axiomatization of the theory of
$(\infty,1)$-categories. Moreover, he showed that the $(\infty, 1)$-category
of autoequivalences of one such theory is the discrete category on the
cyclic group $\Z/2\Z$ of order two, the non-trivial element being the
passage to the opposite $(\infty,1)$-category. More precisely, he computed
the (derived) autoequivalences of the simplicial category obtained as
the Dwyer--Kan localization of complete Segal spaces. A similar calculation
was given by Lurie in \cite[Section 4.4]{LurieGI} using the language of
quasi-categories. These computations imply that any two possibly different
ways of comparing two models for $(\infty, 1)$\nbd-categories differ at most by
the passage to opposites.

\medskip

In this paper, we study the $(\infty, 1)$-category of autoequivalences of
$\infty$\nobreakdash-operads. As in the case of $(\infty,
1)$-categories, there are many different approaches to \oo-operads, including
simplicial operads \cite{CisMoerdDendSimpOper}, $\infty$-operads in the
sense of Lurie \cite[Chapter 2]{LurieHA}, dendroidal sets
\cite{MoerWeissDend, CisMoerdDend} and complete dendroidal Segal spaces
\cite{CisMoerdDendSegal} (which will be called $\Omega$-spaces in this
paper). Again, there are Quillen model categories in the background, and
thanks to recent work of Cisinski--Moerdijk \cite{CisMoerdDend,
CisMoerdDendSegal, CisMoerdDendSimpOper} and
Heuts--Hinich--Moerdijk~\cite{HeutsHinichMoerdijk}, it is known that all
these model structures are connected by Quillen equivalences.

We show that the $(\infty, 1)$-category of autoequivalences of the $(\infty,
1)$-category of \oo-operads is a contractible \oo-groupoid. More precisely,
we prove that the quasi-category of autoequivalences of $\Omega$-spaces is a
contractible Kan complex. This implies that if there is a way to compare two
models for \oo-operads, then this can be done in an essentially unique way.
Similarly, we show that the $(\infty, 1)$-category of autoequivalences of
the $(\infty, 1)$-category of non-symmetric \oo-operads is the discrete
category on the cyclic group $\Z/2\Z$ of order two, the non-trivial element
being the ``mirror autoequivalence''.

One general strategy to compute the autoequivalences of an $(\infty,
1)$\nbd-cate\-gory $\C$, following Toën \cite{ToenAxm}, Lurie \cite[Section
4.4]{LurieGI} and Barwick--Schommer-Pries \cite{BarSchPrUnicity}, is the
following. One first identifies a small category $A$ inside $\C$ such that
\begin{enumerate}
  \item the inclusion functor $A \hookto \C$ is dense;
  \item the autoequivalences of $\C$ restrict to autoequivalences of $A$.
\end{enumerate}
It then follows formally that the autoequivalences of $\C$ sit fully
faithfully in the autoequivalences of $A$. The problem is thus reduced to computing
the autoequivalences of $A$ (up to the question of essential
surjectivity, which is easy in our cases).

If the $(\infty, 1)$-category $\C$ is a localization of an $(\infty,
1)$-category $\qpref{A}$ of simplicial presheaves on a small category $A$,
then one might hope that, in good cases, $A$~would satisfy the two
conditions above. However, in practice it is hard to show the second point
directly. For this purpose, Toën introduced the idea of using an
intermediate (large) category: the so-called $0$-truncated objects of $\C$.
This category is trivially stable under autoequivalences of $\C$. Thus, if we
assume that the objects of $A$ are $0$-truncated, the verification of the
second point is reduced to showing that autoequivalences of $0$-truncated
objects of $\C$ fix the small category $A$. In our cases this turns out to be a much
simpler problem. This strategy is formalized by our
Proposition~\ref{prop:lemma_aut}. (For the case where $\C$ is the
$(\infty, 1)$-category of $(\infty,1)$-categories, Lurie \cite[Section
4.4]{LurieGI} uses a similar strategy, considering the category of posets as
an intermediate (large) category.)

In order to apply this proposition to compute the autoequivalences of
\oo-operads, we need a model defined as a localization of
simplicial presheaves. The only model of this kind for \oo-operads proposed
so far are the $\Omega$\nbd-spaces of Cisinski--Moerdijk
\cite{CisMoerdDendSegal}.  These are defined as a localization of simplicial
presheaves on the small category $\Omega$ of trees introduced by Moerdijk and
Weiss \cite{MoerWeissDend}. It is not hard to identify the $0$-truncated
$\Omega$-spaces: they are the so-called rigid (strict) operads, i.e.,
the operads whose underlying category contains no non-trivial isomorphisms.
It thus suffices to show that autoequivalences of the category of rigid
operads restrict to the small category $\Omega$ and to compute the
autoequivalences of $\Omega$. Most of our section on
\oo-operads is dedicated to the proofs of these two statements. Although
this is not formally needed, we also include a similar computation for the
autoequivalences of (strict) operads.

We also compute the autoequivalences of non-symmetric \oo-operads using the
obvious planar variant of complete dendroidal Segal spaces. The proofs are
quite similar to the symmetric case although the combinatorics differs at
some points. The difference is mainly due to the fact that objects of the
planar version of $\Omega$ have no non-trivial automorphisms.

Finally, we include a calculation of the $(\infty, 1)$\nbd-category of
autoequivalences of $(\infty,n)$-categories. This problem has
already been solved by Barwick and Schommer-Pries
in \cite{BarSchPrUnicity} using a new model for $(\infty, n)$-categories
called $\Upsilon_n$\nbd-spaces. Here we provide an alternative calculation
using instead the model of $\Theta_n$-spaces introduced by Rezk in
\cite{RezkThSp, RezkThSpCorr}, based on the category $\Theta_n$ introduced by
Joyal in~\cite{JoyalTheta}. Our choice of $\Theta_n$-spaces rather than
$\Upsilon_n$-spaces is dictated by the simpler combinatorics of the
category~$\Theta_n$.

\begin{organization}
In Section~\ref{sec:review}, we recall some facts about quasi-categories.
In Section~\ref{sec:autodense}, we study restriction functors induced by
dense functors and we formalize the general strategy for calculating
autoequivalences of certain quasi-categories.  In Section~\ref{sec:oon}, we
prove that the quasi-category of autoequivalences of $\Theta_n$-spaces is
the discrete category $\Zdn$.  Along the way, we compute the
autoequivalences of the categories of strict $n$-categories, rigid strict
$n$-categories and of the category $\Theta_n$.  In
Section~\ref{sec:operads}, we prove that the quasi-category of
autoequivalences of $\Omega$-spaces is a contractible Kan complex. We also
calculate the autoequivalences of the categories of operads, rigid operads
and of the category $\Omega$. Finally, in Section~\ref{sec:nsoperads}, we
turn to the quasi-category of planar $\Omega$-spaces and show that its
quasi-category of autoequivalences is the discrete category $\Z/2\Z$. We
also describe the autoequivalences of non-symmetric operads, non-symmetric
rigid operads and planar trees.
\end{organization}

\begin{notation}
If $A$ is a small category, we will denote by $\pref{A}$ the category of
(set-valued) presheaves on $A$ and by $\spref{A}$ the category of
simplicial presheaves on $A$. If $\C$ is a category, we will denote by
$\aut(\C)$ the \emph{set} of autoequivalences of $\C$ and by $\Aut(\C$) the
\emph{category} of autoequivalences of $\C$. The set
$\aut(\C)$ and the category $\Aut(\C)$ will sometimes be considered as
a monoid and a strict monoidal category, respectively, the additional
structure being given by composition.  We will say that a morphism of a
category (or more generally of a strict $n$-category) is non-trivial if it
is not an identity.

If $\C$ and $\D$ are quasi-categories, we will denote by $\Fun(\C, \D)$
the quasi-category of functors from $\C$ to $\D$. The full subcategory of
$\Fun(\C, \C)$ spanned by the equivalences will be denoted by $\Aut(\C)$.

We will assume for simplicity (contrary to the strong opinions of the first
two authors) that our model categories have functorial factorizations. If
$\M$ is a model category and $S$ is a class of morphisms of $\M$, we will
denote by $S^{-1}\M$ the left Bousfield localization of $\M$ with respect to
$S$ (if it exists).

We will neglect the usual set theoretic issues related to category theory.
In particular, we will shamelessly apply the nerve functor to non-small
categories.
\end{notation}

\section{Review of quasi-categories}
\label{sec:review}

In this section, we recall some facts about quasi-categories, mostly about
the relation between model categories and quasi-categories, and about
localizations of locally presentable quasi-categories. We assume that the
reader is familiar with the basics of the theory of quasi-categories as
developed in the foundational work of Joyal \cite{JoyalQCatKan, JoyalNotes,
JoyalQCatAppl} and Lurie \cite{LurieHTT, LurieHA}. For an introduction to
this theory emphasizing the philosophy, see \cite{Groth}.

\begin{paragrp}
We will denote by $N \colon \Cat \to \SSet$ the nerve functor from (small)
categories to simplicial sets. Since the nerve of a category is a
quasi-category, this functor induces a fully faithful functor from
categories to quasi-categories. We will often consider this functor as an
inclusion.
\end{paragrp}

\begin{defip}[Lurie]
The \ndef{underlying quasi-category} of a model category $\M$ is a
quasi-category $\U(\M)$ endowed with a functor \hbox{$f\colon \M \to\U(\M)$}
such that for any quasi-category~$\D$, the induced functor
\[
f^\ast\colon\Fun(\U(\M),\D)\longto\Fun(\M,\D)
\]
is fully faithful with essential image the functors $\M\to\D$ which send
weak equivalences of $\M$ to equivalences in~$\D$. This quasi-category
$\U(\M)$, if it exists (and it does, see the next proposition), is
determined uniquely up to equivalence of quasi-categories.
\end{defip}

\begin{remp}
This definition differs slightly from the original definition of Lurie
\cite[Definitions 1.3.4.1 and 1.3.4.15]{LurieHA} in which one restricts to
cofibrant objects of $\M$. Nevertheless, the two definitions are equivalent
by \cite[Remark 1.3.4.16]{LurieHA}.
\end{remp}

\begin{remp}
The definition of $\U(\M)$ only depends on the underlying category of~$\M$
and the weak equivalences of $\M$. In particular, if $\M$ and $\N$ are two model
categories on the same underlying category with same weak equivalences,
then $\U(\M)$ and $\U(\N)$ are canonically equivalent.
\end{remp}

\begin{propp}[Lurie]
Every model category has an underlying quasi-category.
\end{propp}

\begin{proof}
This follows from \cite[Remark 1.3.4.2]{LurieHA}.
\end{proof}

\begin{paragrp}
We will denote by $\cN \colon \sCat \to \SSet$ Cordier's coherent nerve functor
from simplicial categories to simplicial sets (see \cite[Definition
1.1.5.5]{LurieHTT}). If $\M$ is a simplicial model category, we will denote
by $\M^\circ$ the full simplicial subcategory of $\M$ spanned by the cofibrant and
fibrant objects. This simplicial category $\M^\circ$ is locally fibrant in
the sense that all its mapping spaces are Kan complexes.  It follows from
\cite[Theorem 2.1]{cordierporter_vogt} that its coherent nerve
$\cN(\M^\circ)$ is a quasi-category.
\end{paragrp}

\begin{thmp}[Lurie]\label{thm:under_simp}
If $\M$ is a simplicial model category, then $\cN(\M^\circ)$ is the
underlying quasi-category of $\M$.
\end{thmp}

\begin{proof}
This is \cite[Theorem 1.3.4.20]{LurieHA}.
\end{proof}

\begin{paragrp}
Recall from \cite[Section 5.5]{LurieHTT} that the classical notion of
locally presentable category can be generalized to the notion of locally
presentable quasi-category. By \cite[Proposition A.3.7.6]{LurieHTT}
and \cite[Proposition 1.3.4.22]{LurieHA}, these quasi-categories can be
characterized as those being the underlying quasi-category of a
combinatorial model category.
\end{paragrp}

\begin{paragrp}
Denote by $\mathcal{S}$ the quasi-category of spaces, that is, the
underlying quasi-category of the Kan--Quillen model structure on simplicial
sets. If $A$ is a small category, then the \ndef{quasi-category of presheaves}
$\qpref{A}$ on $A$ is the quasi-category $\Fun(A^\op, \mathcal{S})$.
\end{paragrp}

\begin{propp}[Heller, Bousfield--Kan]
Let $A$ be a small category. We have the following two simplicial proper
combinatorial model structures on the category~$\spref{A}$ of simplicial
presheaves on~$A$:
\begin{enumerate}
  \item the \ndef{injective model structure}, whose weak equivalences and
    cofibrations are the objectwise weak equivalences and the objectwise
    cofibrations, respectively;
  \item the \ndef{projective model structure}, whose weak equivalences and
    fibrations are the objectwise weak equivalences and the objectwise
    fibrations, respectively.
\end{enumerate}
\end{propp}

\begin{proof}
It seems that the first appearances of the injective and projective model
structures are \cite[Theorem 4.5]{Heller} and \cite[Chapter XI, \S 8]{BK}, respectively.
See \cite[Proposition A.2.8.2]{LurieHTT} for a more general statement.
The fact that these model structures are simplicial follows easily from the
fact that the Kan--Quillen model structure on simplicial sets is simplicial.
The left properness is obvious for the injective model structure and right
properness for the projective model structure follows easily from right
properness of the Kan--Quillen model structure. Since left and right
properness only depend on the class of weak equivalences, it follows that
these structures are both proper.
\end{proof}

We will denote by $\spref{A}_\inj$ and $\spref{A}_\proj$ these two model
structures.

\begin{propp}[Lurie]\label{prop:desc_PA}
The quasi-category $\qpref{A}$ of presheaves on a small category $A$ is
canonically equivalent to the underlying quasi-category of the projective
model structure on $\spref{A}$.
\end{propp}

\begin{proof}
This is a special case of \cite[Proposition~1.3.4.25]{LurieHA}.
\end{proof}

\begin{paragrp}
If $\C$ is a quasi-category and $X$, $Y$ are two objects of $\C$, we will
denote by $\Map_\C(X, Y)$ the space of morphisms from $X$ to $Y$. See
\cite[Section~1.2.2]{LurieHTT} or \cite{DuggerSpivakMap} for various
approaches to define this space. Recall from \cite[Section 2.2]{LurieHTT}
that if $\M$ is a simplicial model category, then the mapping spaces of the
underlying quasi-category of $\M$ can be computed using the simplicial
enrichment of $\M$ (when restricted to cofibrant fibrant objects).
\end{paragrp}

\begin{paragrp}
Let $\C$ be a quasi-category and let $S$ be a class of morphisms of $\C$.
An object~$Y$ of $\C$ is \emph{$S$-local} if for every map $f\colon X\to X'$
in $S$, the induced map
\[
f^\ast\colon\Map_{\C}(X',Y)\longto\Map_{\C}(X,Y)
\]
is a weak equivalence. The full subcategory of the quasi-category $\C$
spanned by the $S$-local objects is called the \ndef{localization of $\C$ by
$S$}. We will denote it by $S^{-1}\C$.
\end{paragrp}

\begin{propp}[Lurie]
If $\C$ is a locally presentable quasi-category and $S$ is a \emph{set} of morphisms of
$\C$, then the inclusion $i \colon S^{-1}\C \hookto \C$ admits a left adjoint $L$. In
other words, we have a reflective localization
\[
L\colon \C\myrightleftarrows{\rule{0.4cm}{0cm}} S^{-1}\C\colon i.
\]
\end{propp}

\begin{proof}
This is \cite[Proposition 5.5.4.15.(3)]{LurieHTT}.
\end{proof}

\begin{remp}
When the quasi-category $\C$ is an ordinary category $A$, the mapping space
$\Map_A(X, Y)$ is simply the discrete simplicial set $A(X, Y)$. In
particular, an object $Y$ in $A$ is $S$-local if and only if, for all
$f\colon X\to X'$ in $S$, the induced map
\[
f^\ast\colon A(X',Y)\longto A(X,Y)
\]
is a bijection, or, in other words, if and only if $Y$ is right orthogonal
to $S$.
\end{remp}

\begin{propp}\label{prop:loc_mcat_qcat}
Let $\M$ be a left proper combinatorial model category and let $S$ be a set
of maps of $\M$. There is a canonical equivalence of quasi-categories
\[
S^{-1}(\U(\M)) \stackrel{\cong}{\longto} \U(S^{-1} \M).
\]
\end{propp}

\begin{proof}
The functor $\U(\M)\to\U(S^{-1}\M)$ sends $S$ to equivalences and hence, by the
universal property of the localization, we obtain a functor
\[
\phi\colon S^{-1}(\U(\M)) \longto \U(S^{-1} \M).
\]
Let us show that this functor is an equivalence.

We first assume that $\M$ is a \emph{simplicial} left proper combinatorial
model category. By Theorem~\ref{thm:under_simp}, the underlying
quasi-categories of $\M$ and $S^{-1}\M$ are $\cN(\M^\circ)$ and
$\cN((S^{-1}\M)^\circ)$, respectively. The two quasi-categories
$S^{-1}\cN(\M^\circ)$ and $\cN((S^{-1}\M)^\circ)$ sit fully
faithfully in~$\cN(\M^\circ)$. Moreover, using the compatibility between
the mapping spaces of the model category $\M$ and those of the
quasi-category $\cN(\M^\circ)$, it is easy to check that these two
quasi-categories are equal as subcategories of $\cN(\M^\circ)$. The identity
functor $S^{-1}\cN(\M^\circ) \to \cN((S^{-1}\M)^\circ)$ is easily seen to be
canonically equivalent to the functor $\phi$, thereby concluding the proof
of the simplicial case.

We now want to drop the additional assumption that $\M$ is simplicial by an
application of a well-known result of Dugger. For that purpose, let us
assume that we have a Quillen equivalence $\M_1\rightleftarrows\M_2$ between
left proper combinatorial model categories and a set $S_2$ of maps of
$\M_2$. Let $S_1$ be the derived image of $S_2$ under the right adjoint of the
Quillen pair. By \cite[Theorem 3.3.20]{hirschhorn:loc}, there is an
induced Quillen equivalence $(S_1)^{-1}\M_1\rightleftarrows (S_2)^{-1}\M_2$.
Recall that Quillen
equivalences between combinatorial model categories induce equivalences
between the underlying quasi-categories \cite[Lemma~1.3.4.21]{LurieHA}. We
thus get a commutative diagram of quasi-categories
\[
\xymatrix{
\U(\M_1)\ar[r]\ar[d]_-\simeq&(S_1)^{-1}\U(\M_1)\ar[r]^-{\phi_1}\ar[d]_-\simeq&
\U((S_1)^{-1}\M_1)\ar[d]^-\simeq\\
\U(\M_2)\ar[r]&(S_2)^{-1}\U(\M_2)\ar[r]_-{\phi_2}&\U((S_2)^{-1}\M_2),
}
\]
in which the vertical maps are equivalences. This implies that $\phi_1$
is an equivalence if and only if $\phi_2$ is an equivalence.

Finally, let $\M$ be a left proper combinatorial model category. Then by
\cite{dugger:universal}, there is a left Quillen equivalence $\N\to\M$ such
that $\N$ is a simplicial left proper combinatorial model category. The
statement now follows easily from the previous two paragraphs.
\end{proof}

\begin{paragrp}
Let $\C$ be a quasi-category. An object $Y$ of $\C$ is said to be
\ndef{$0$-truncated} if for every object $X$ of $\C$, the mapping space
$\Map_\C(X, Y)$ is discrete. We will denote by~$\ztr{\C}$ the full
subcategory of $\C$ spanned by the $0$-truncated objects of $\C$.
\end{paragrp}

\begin{propp}[Lurie]
If $\C$ is a locally presentable quasi-category, then the inclusion $i
\colon \ztr{\C} \hookto
\C$ admits a left adjoint $L$. In other words, we have a reflective
localization
\[
L\colon\C \myrightleftarrows{\rule{0.4cm}{0cm}}\ztr{\C} \colon i.
\]
\end{propp}

\begin{proof}
This is \cite[Proposition 5.5.6.18]{LurieHTT}.
\end{proof}

\section{Autoequivalences and dense functors}
\label{sec:autodense}

In this section, we formalize the general strategy described in the
introduction for computing certain quasi-categories of autoequivalences.
This is made precise by Proposition~\ref{prop:lemma_aut} and will be used in
the following three sections to determine the autoequivalences of $(\infty,
n)$-categories, \oo-operads and non-symmetric \oo-operads.

\begin{paragrp}
Let $A$ be a small category. By Proposition~\ref{prop:desc_PA}, the
quasi-category $\qpref{A}$ is canonically equivalent to $\cN(\sprefpo{A})$.
Let us choose a functorial cofibrant replacement functor $Q$ for
$\sprefp{A}$ with a natural trivial fibration $\phi\colon Q \to \id{}$.
Since every discrete presheaf is fibrant in $\sprefp{A}$, applying the
functor $Q$ to such a presheaf yields a cofibrant fibrant object of
$\sprefp{A}$. Thus the functor~$Q$ induces a morphism of simplicial
categories $\pref{A} \to \sprefpo{A}$ and hence a morphism of
quasi-categories $\pref{A} \to \qpref{A}$.
\end{paragrp}

\begin{propp}\label{prop:0-trunc-psh}
Let $A$ be a small category. Then the functor $\pref{A} \to \qpref{A}$
is fully faithful, factors through the subcategory $\ztr{\qpref{A}}$ of
$0$-truncated objects, and the restricted functor $\pref{A}
\to\ztr{\qpref{A}}$ is an equivalence of quasi-categories.
\end{propp}

\begin{proof}
We begin by showing that the morphism $\pref{A}\to\qpref{A}$ is fully
faithful. Denote by $\delta \colon \pref{A} \to \spref{A}$ the functor sending a
presheaf to the associated discrete simplicial presheaf. We have to show
that if $X$ and $Y$ are two presheaves on~$A$, then the map
\[
Q \circ \delta \colon \Map_{\pref{A}}(X, Y) \longto \Map_{\spref{A}}(Q\delta X, Q\delta Y)
 \]
is a weak equivalence. Since the functor $\delta$ is fully faithful, this
amounts to saying that the map
\[ \Map_{\spref{A}}(\delta X, \delta Y) \longto \Map_{\spref{A}}(Q\delta X,
  Q\delta Y) \]
induced by $Q$ is a weak equivalence. But this map sits in the commutative diagram
\[
\xymatrix{
\Map_{\spref{A}}(\delta X,\delta
Y)\ar[d]_Q\ar[r]^-\cong&\Map_{\pref{A}}(\pi_0\delta X,
Y)\ar[dd]^-{\pi_0(\phi_{\delta X})^\ast}\\
\Map_{\spref{A}}(Q\delta X,Q\delta Y)\ar[d]_-{(\phi_{\delta Y})_\ast}&\\
\Map_{\spref{A}}(Q\delta X, \delta
Y)\ar[r]_-\cong&\Map_{\pref{A}}(\pi_0Q\delta X,Y),
}
\]
where the top and the bottom maps are isomorphisms coming from the
adjunction~$(\pi_0, \delta)$.
Since $Q\delta X$ is cofibrant and $Q\delta Y \to \delta Y$ is a trivial
fibration, the map~$(\phi_{\delta Y})_\ast$ is a weak equivalence.
The map $\pi_0(\phi_{\delta X})^\ast$ is even an isomorphism. It follows
that the functor of the statement is fully faithful.

The fully faithfulness of the functor $\pref{A}\to\qpref{A}$ easily implies
that it factors through $0$-truncated objects and it remains to show that
the induced functor is essentially surjective. Given $X$ in
$\ztr{\qpref{A}}$, we claim that the canonical map $\eta\colon X\to \delta
\pi_0(X)$ is a natural equivalence. Obviously, $\eta$ induces an isomorphism
on~$\pi_0$. To conclude, it thus suffices to show that for every object $a$
in $A$, the space~$X_a$ is discrete. But we have
$\Map_{\spref{A}}(a,X)\cong X_a$, where $a$ denotes the discrete simplicial
presheaf associated to $a$. This shows the result since $X$ is $0$\nbd-truncated.
\end{proof}

\begin{remp}
It follows from the previous proposition that if $S$ is a set of morphisms
of $\pref{A}$, then an object of $\pref{A}$ is $S$-local (i.e., right
orthogonal with respect to~$S$) if and only if it is $S$-local when
considered as an object of $\qpref{A}$.
\end{remp}

\begin{propp}\label{prop:0-trunc-loc}
Let $\C$ be a locally presentable quasi-category and let $S$ be a set of
morphisms between $0$-truncated objects of~$\C$. Then the quasi-categories
$S^{-1}(\ztr{\C})$ and $\ztr{(S^{-1}\C)}$ are equal as subcategories of $\C$.
\end{propp}

\begin{proof}
We have the following compositions of reflective localizations
\[
\C \myrightleftarrows{\rule{0.4cm}{0cm}}S^{-1}\C
\myrightleftarrows{\rule{0.4cm}{0cm}}\tau_{\leq 0}(S^{-1}\C),
\qquad
\C \myrightleftarrows{\rule{0.4cm}{0cm}} \tau_{\leq 0}\C
\myrightleftarrows{\rule{0.4cm}{0cm}} S^{-1}(\tau_{\leq 0}\C),
\]
and it hence suffices to show that the two quasi-categories
$\tau_{\leq 0}(S^{-1}\C)$ and $S^{-1}(\tau_{\leq 0}\C)$ have the same
objects. But if $Y$ is an object of $S^{-1}(\tau_{\leq 0}\C)$, then
$\Map_{\C}(X,Y)$ is discrete for all $X$ in~$\C$. In particular,
$\Map_{S^{-1}\C}(X,Y)=\Map_{\C}(X,Y)$ is discrete for all $X$ in
$S^{-1}\C$ and hence $Y$ lies in $\tau_{\leq 0}(S^{-1}\C)$. Conversely,
if $Y$ is an object of $\tau_{\leq 0}(S^{-1}\C)$, then
$\Map_{S^{-1}\C}(X,Y)$ is discrete for all $X$ in $S^{-1}\C$. Given an
arbitrary $X$ in $\C$, then, using the localization functor $L\colon\C\to
S^{-1}\C$ and the $S$-locality of~$Y$, the mapping space $\Map_{\C}(X,Y)$ is
weakly equivalent to the discrete space $\Map_{S^{-1}\C}(LX,Y)$, and so $Y$
lies in $S^{-1}(\tau_{\leq 0}\C)$.
\end{proof}

\begin{corop}\label{coro:0-trunc}
Let $A$ be a small category and let $S$ be a set of morphisms of~$\pref{A}$.
The morphism $\pref{A} \to \qpref{A}$ induces an equivalence
from the category $S^{-1}\pref{A}$ to the quasi-category
$\ztr{(\loc{\qpref{A}}{S})}$.
\end{corop}

\begin{proof}
By Proposition~\ref{prop:0-trunc-psh}, the morphism $\pref{A} \to
\qpref{A}$ induces an equivalence $\pref{A} \to \ztr{\qpref{A}}$ and hence
an equivalence $S^{-1}\pref{A} \to S^{-1}(\ztr{\qpref{A}})$. But by the
previous proposition, $S^{-1}(\ztr{\qpref{A}})$ is nothing but
$\ztr{(S^{-1}\qpref{A})}$.
\end{proof}

\begin{paragrp}\label{paragr:def_dense}
We will say that a functor $f\colon\C\to\D$ between quasi-categories is
\emph{dense} if the identity transformation $f\to f$ exhibits
$\id{}\colon\D\to\D$ as a left Kan extension of~$f$ along~$f$.
(In Lurie's terminology, see \cite[Definition~4.4.2]{LurieGI}, one says that
$f$ strongly generates $\D$.) When $\C$ is small and $\D$ is cocomplete, this
amounts to saying that the associated nerve functor (that is, the right
adjoint to the canonical functor $\qpref{\C} \to \D$) is fully faithful
\cite[Remark~4.4.4]{LurieGI}. In particular, if $A$ is a small category,
then the Yoneda embedding $A \to \qpref{A}$ is dense. More generally, if $A$
is a small category and $S$ is a set of morphisms of $\qpref{A}$, then the
canonical functor $A \to S^{-1}\qpref{A}$ is dense.
\end{paragrp}

\begin{propp}[Lurie]\label{prop:auto_dense}
Let $i\colon\A \to \B$ be a dense inclusion of quasi-categories such that $\A$ is
small and $\B$ admits small colimits. If any autoequivalence of $\B$ restricts to
an autoequivalence of $\A$, then the inclusion $i$ induces a fully faithful functor
$i^\ast\colon\Aut(\B)\to\Aut(\A)$.
\end{propp}

\begin{proof}
Denote by $\Fun^L(-,-)$ the quasi-category of colimit-preserving functors
between two quasi-categories. The dense inclusion $i$ induces a fully
faithful functor $i^\ast \colon \Fun^L(\B,\B)\to\Fun(\A,\B)$ by \cite[Remark
4.4.5]{LurieGI}. Consider the diagram
\[
\xymatrix{
\Fun^L(\B,\B)\ar[r]^-{i^\ast}&\Fun(\A,\B)\\
\qAut(\B)\ar[u]\ar[ru]\ar@{.>}[r]_-{i^\ast}&\qAut(\A),\ar[u]
}
\]
where the left vertical map is the canonical inclusion and the right
vertical map is the fully faithful map given by postcomposing with $i$.
By our assumption on autoequivalences of $\B$, the diagonal map factors over
$\qAut(\A)$, giving rise to the dotted morphism. Since all remaining
maps are fully faithful, the same is true for
$i^\ast\colon\qAut(\B)\to\qAut(\A)$.
\end{proof}

\begin{propp}\label{prop:lemma_aut}
Let $A$ be a small category and let $S$ be a set of morphisms of~$\pref{A}$.
Assume that the following conditions are satisfied:
\begin{enumerate}
  \item Representable presheaves in $\pref{A}$ are $S$-local.
  \item Any autoequivalence of $S^{-1}\pref{A}$ restricts to an autoequivalence
    of $A$ (the category $A$ being included in $S^{-1}\pref{A}$ because of
    {\rm (i)}).
\end{enumerate}
Then $A \to S^{-1}\qpref{A}$ induces a fully faithful functor $\Aut(S^{-1}\qpref{A})
\to \Aut(A)$. In particular, the quasi-category $\Aut(S^{-1}\qpref{A})$ is discrete.
\end{propp}

\begin{proof}
By Corollary~\ref{coro:0-trunc}, we know that there is an equivalence of
quasi-categories $\ztr{(S^{-1}\qpref{A})}\cong S^{-1}\pref{A}$. Since
$0$-truncated objects are stable under equivalences, we obtain a functor
$\qAut(S^{-1}\qpref{A}) \to \Aut(S^{-1}\pref{A}) $. Using assumption~(ii),
we get a functor $\Aut(S^{-1}\pref{A})\to\Aut(A)$. Obviously, the
composition
\[ \qAut(S^{-1}\qpref{A}) \longto \Aut(S^{-1}\pref{A})\longto\Aut(A) \]
of these two functors is induced by $A \to S^{-1}\pref{A}\to
S^{-1}\qpref{A}$. But this functor is dense (see the end of
paragraph~\ref{paragr:def_dense}) and the result thus follows from the
previous proposition.
\end{proof}

\section{Autoequivalences of the \pdfinftyo-category of \pdfinftyn-categories}
\label{sec:oon}

The aim of this section is to show that the quasi-category of
autoequivalences of the quasi-category of $\Theta_n$-spaces, which is a
model for $(\infty, n)$-categories, is the discrete category $\Zdn$.  This
calculation is a consequence of two results of Barwick and Schommer-Pries
\cite{BarSchPrUnicity}, namely the computation of the autoequivalences of
$\Upsilon_n$-spaces and the comparison of $\Upsilon_n$-spaces and
$\Theta_n$-spaces. Here we give a direct proof of this fact using the
general strategy outlined in the introduction and formalized by
Proposition~\ref{prop:lemma_aut}. In particular, we reduce to the
computation of the autoequivalences of Joyal's category~$\Theta_n$.

\smallbreak

Throughout the section, we fix $n$ such that $0 \le n \le \infty$.

\subsection{Preliminaries on strict $n$-categories and the category
$\Theta_n$}

\begin{paragr}
We will denote by $\Cat$ the category of small categories and by $\nCat{n}$
the category of small strict $n$-categories. Recall that these categories
are complete and cocomplete. We refer the reader to \cite{AraStrWeak} for
details on strict $n$-categories compatible with the notation we will use in
this section.
\end{paragr}

\begin{paragr}
Let $k$ be such that $0 \le k \le n$. A strict $k$-category can be
considered as a strict $n$-category whose $i$-arrows are identities for $i >
k$. This defines a fully faithful functor $\nCat{k} \hookto \nCat{n}$ which
we will consider as an inclusion. In particular, we will say that a strict
$n$-category is a $k$-category if it is in the image of this functor.

This inclusion functor admits both a left adjoint and a right adjoint. In this
paper, we will only consider the right adjoint. It sends a strict
$n$-category $C$ to the strict $k$-category $\tr_k(C)$ obtained by throwing
out the $i$-arrows of $C$ for $i > k$. This $k$-category $\tr_k(C)$ will be
called the \ndef{$k$-truncation} of $C$.
\end{paragr}

\begin{paragr}
The \ndef{$k$-disk} $\Dn{k}$, where $k \ge 0$, is the strict \oo-category
corepresenting the functor ``set of $k$-arrows'' $\Ar_k \colon \nCat{\infty}
\to \Set$. The strict \oo-category $\Dn{k}$ is actually a $k$-category. Here are
pictures of (the underlying \oo-graphs without the identities of) $\Dn{k}$
in low dimension:
\[
\UseAllTwocells
\Dn{0} = \bullet
\;,\quad
\Dn{1} = \xymatrix{\bullet \ar[r] & \bullet}
\;,\quad
    D_2 = \xymatrix@C=3pc@R=3pc{\bullet \ar@/^2.5ex/[r]_{}="0"
    \ar@/_2.5ex/[r]_{}="1"
    \ar@2"0";"1"
      &  \bullet}
\quad \text{and}\quad
    D_3 = \xymatrix@C=3pc@R=3pc{\bullet \ar@/^3ex/[r]_(.47){}="0"^(.53){}="10"
    \ar@/_3ex/[r]_(.47){}="1"^(.53){}="11"
    \ar@<2ex>@2"0";"1"_{}="2" \ar@<-2ex>@2"10";"11"^{}="3"
    \ar@3"3";"2"_{}
    & \bullet \pbox{.}}
  \]
For $k > 0$, we have two morphisms $\sigma, \tau \colon \Dn{k-1} \to
\Dn{k}$ corepresenting the natural transformations source and target $\Ar_k
\to \Ar_{k-1}$, respectively. Concretely, $\sigma$~(respectively $\tau$) sends the
unique non-trivial $(k-1)$-arrow of $\Dn{k-1}$ to the source (respectively
to the target) of the unique non-trivial $k$-arrow of $\Dn{k}$.

For $k > k' \ge 0$, we will also denote by $\sigma$ and $\tau$ the morphisms
$\Dn{k'} \to \Dn{k}$ obtained by composing
\[
  \Dn{k'} \stackrel{\sigma}{\longto} \Dn{k'+1} \cdots \Dn{k-1} \stackrel{\sigma}{\longto}
  \Dn{k}
  \quad\text{and}\quad
  \Dn{k'} \stackrel{\tau}{\longto} \Dn{k'+1} \cdots \Dn{k-1} \stackrel{\tau}{\longto}
  \Dn{k},
\]
respectively. Note that $\sigma, \tau \colon \Dn{k'} \to \Dn{k}$ are the
only monomorphisms $\Dn{k'} \hookto \Dn{k}$.
\end{paragr}

\begin{paragr}\label{paragr:def_table}
A \ndef{table of dimensions} is a table
\[ \tabdim, \]
where $m \ge 1$, filled with non-negative integers satisfying
$k_i > k'_i$ and $k_{i+1} > k'_i$
for every $i$ such that $0 < i < m$. The greatest integer appearing in the
table is called the \ndef{height} of the table.

To such a table $T$, we can associate the following diagram
\[
\xymatrix@R=.2pc@C=1pc{
\Dn{k_1} &  & \Dn{k_2} &  & \Dn{k_3} &        & \Dn{k_{m - 1}} & & \Dn{k_{m}} \\
  &  &   &  &   & \cdots &     & & \\
  & \Dn{k'_1}
  \ar[uul]^\sigma
  \ar[uur]_\tau   &   &
  \Dn{k'_2}
  \ar[uul]^\sigma
  \ar[uur]_\tau  &  &  & &
\Dn{k'_{m-1}} \ar[uul]^\sigma \ar[uur]_\tau
& &
}
\]
in $\nCat{\infty}$. We will denote this diagram by $\D_T$. The colimit of
$\D_T$ will be denoted by $\Theta(T)$. We will also sometimes denote it by
\[
\Dn{k_1} \amalg_{\Dn{k'_1}} \dots \amalg_{\Dn{k'_{m-1}}} \Dn{k_m}.
\]
This is an $n$-category, where $n$ is the height of $T$.
Note that the diagram $\D_T$ actually comes from a diagram of $n$-graphs and
it follows that the $n$-category $\Theta(T)$ is the free strict $n$-category
on an $n$-graph $\Theta_0(T)$.

It is easy to see that the strict \oo-category $\Theta(T)$ does not admit
any non-trivial automorphisms. This means, in particular, that there is a
unique cocone making~$\Theta(T)$ the colimit of the diagram~$\D_T$.
\end{paragr}

\begin{exam}
If $T$ is the following table of dimensions
\[
\setcounter{MaxMatrixCols}{20}
\left(
\begin{matrix}
2 && 2 && 2 && 3 && 2 && 1 \\
& 1 && 0 && 1 && 1 && 0
\end{matrix}
\right),
\]
then $\Theta_0(T)$ is the following $3$-graph:
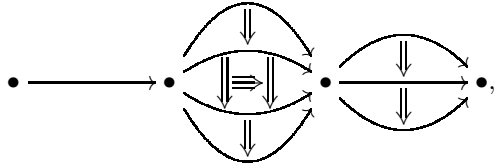
\[
\xymatrix@C=4pc@H=5pc{
\bullet
\ar[r]
&
\bullet
\ar@/^2.5pc/[r]|{\vphantom{X}}="g"
\ar@/^1pc/[r]|{\vphantom{X}}="h"
\ar@/_1pc/[r]|{\vphantom{X}}="i"
\ar@/_2.5pc/[r]|{\vphantom{X}}="j"
\ar@{=>}"g";"h"
\ar@<-2ex>@{=>}"h";"i"_{}="0"
\ar@<2ex>@{=>}"h";"i"_{}="1"
\ar@3"0";"1"
\ar@{=>}"i";"j"
&
\bullet
\ar@/^1.5pc/[r]|{\vphantom{X}}="k"
\ar[r]|{\vphantom{X}}="l"
\ar@/_1.5pc/[r]|{\vphantom{X}}="m"
\ar@{=>}"k";"l"
\ar@{=>}"l";"m"
&
\bullet \pbox{,}
&
}
\]
and the $3$-category $\Theta(T)$ associated to $T$ is the strict
$3$-category freely generated by this $3$-graph.
\end{exam}

\begin{paragr}
Joyal's category $\Theta_n$ of $n$-cells is defined in the following way. The
objects of~$\Theta_n$ are the tables of dimensions of height at most
$n$. If $S$ and $T$ are two objects of $\Theta_n$, then the set of
morphisms from $S$ to $T$ is given by
\[ \Theta_n(S, T) = \nCat{n}(\Theta(S), \Theta(T)). \]
By definition, we have a fully faithful functor $\Theta_n \to \nCat{n}$. It
is easy to see that this functor is also injective on objects and we will
consider it as an inclusion. In particular, we will not make a difference
between the table $T$ and the associated strict \oo-category $\Theta(T)$.
\end{paragr}

\begin{rem}
  Here are some comments on $\Theta_n$ for $n = 0, 1, \infty$:
  \begin{enumerate}
    \item The category $\Theta_0$ is the terminal category. It is \emph{not}
      the category $\Theta_0$ introduced by Berger in \cite{BergerNerve},
      which corresponds to $\Theta_{\infty, 0}$ with the notation of
      our paragraph~\ref{paragr:active_inert}.
    \item The category $\Theta_1$ is canonically isomorphic to the category
      $\Delta$ of simplices.
    \item The category $\Theta_\infty$ is canonically isomorphic to the
      category $\Theta$ introduced by Joyal in \cite{JoyalTheta} by
      combinatorial means.
  \end{enumerate}
\end{rem}

\begin{rem}
The category $\Theta_n$ has a universal property relating it to strict
$n$\nbd-cate\-gories. Roughly speaking, it is the free category (having
certain colimits) endowed with a strict $n$-cocategory object. See Propositions
3.11 and 3.14 of \cite{AraThtld} for the case~$n = \infty$.
\end{rem}

\begin{paragr}\label{paragr:active_inert}
We define a category $\Theta_{n, 0}$ in the following way:
the objects of $\Theta_{n,0}$ are the same as the ones of $\Theta_n$ and the
set of morphisms in $\Theta_{n, 0}$ from an object $S$ to an object $T$ is
given by
\[
\Theta_{n, 0}(S, T) = \nGraph(\Theta_0(S), \Theta_0(T)),
\]
where $\nGraph$ denotes the category of $n$-graphs. The free strict
$n$-category functor induces a canonical functor from $\Theta_{n, 0}$ to
$\Theta_n$ which is obviously faithful.

We will say that a morphism of $\Theta_n$ is \ndef{inert} if it comes from a
morphism of~$\Theta_{n, 0}$. It is easy to see that inert morphisms are
monomorphisms. A morphism $S \to T$ of~$\Theta_n$ is said to be
\ndef{active} if it does not factor through any non-trivial inert morphism
$T' \to T$, that is, if every time it can be written as $if$, where $i$ is
inert, then $i$ is an identity.

Active morphisms can be described more concretely in the following way. Let
$f \colon \Dn{k} \to T$ be a morphism of $\Theta_n$. Such a morphism corresponds
to a $k$-arrow of~$\Theta(T)$. This $k$-arrow can be expressed using
compositions and identities from some generators of the free strict
$n$-category $\Theta(T)$, that is, from the arrows of~$\Theta_0(T)$. It is
not hard to see that the morphism $f$ is active if and only if all the
generators of~$\Theta(T)$ are needed to express this $k$-arrow. In other
words, $\Dn{k} \to T$ is active if and only if it corresponds to (an
identity in dimension $k$ of) the total composition of $\Theta(T)$. This
means that such an active morphism exists if and only if $k$ is greater than or
equal to the height of $T$, and that in this case, it is unique.

More generally, a morphism $S \to T$ of $\Theta_n$ corresponds to a pasting
scheme of shape $\Theta_0(S)$ in $\Theta(T)$ and it is active if and only if
all the generators of $\Theta(T)$ are needed to express the arrows of this
pasting scheme.
\end{paragr}

\begin{prop}[Berger, Weber]\label{prop:active_inert}
Every morphism of $\Theta_n$ can be written in a unique way as a composition
of an active morphism followed by an inert morphism.
\end{prop}

\begin{proof}
This follows from the general machinery of \cite{Weber} (see Example 4.21).
This was first proved in \cite{BergerNerve} (see Lemma 1.11) using a
different but equivalent definition of $\Theta_n$. See also
\cite[Proposition~3.3.10]{AraThesis}.
\end{proof}

\begin{paragr}
The category of \ndef{$n$-cellular sets} is the category $\pref{\Theta_n}$
of presheaves on $\Theta_n$. The inclusion $\Theta_n \hookto
\nCat{n}$ induces a functor $N_n \colon \nCat{n} \to \pref{\Theta_n}$ sending a strict
$n$-category $C$ to the $n$-cellular set $T \mapsto \nCat{n}(T, C)$. This
functor $N_n$ is called the \ndef{$n$-cellular nerve}.
\end{paragr}

\begin{paragr}\label{paragr:spine_nCat}
Let
\[
T = \tabdimk
\]
be an object of $\Theta_n$. By definition, we have
\[
T = \Dn{k_1} \amalg_{\Dn{k'_1}} \dots \amalg_{\Dn{k'_{m-1}}} \Dn{k_m}
\]
in $\Theta_n$. The \ndef{spine} $I_T$ of $T$ is the $n$-cellular set
\[
I_T = \Dn{k_1} \amalg_{\Dn{k'_1}} \dots \amalg_{\Dn{k'_{m-1}}} \Dn{k_m},
\]
where the colimit is now taken in the category $\pref{\Theta_n}$.
There is a canonical morphism
\[
i_T \colon I_T \longto T.
\]
It is not hard to check that this morphism is a monomorphism.
We will denote by~$\I$ the set
\[ \I = \{i_T\mid T \in \Ob(\Theta_n)\} \]
of spine inclusions.

Let now $X$ be an $n$-cellular set. For any object $T$ of $\Theta_n$, the
map $i_T$ induces a \emph{Segal map}
\[
X(T) \cong \pref{\Theta_n}(T, X) \longto \pref{\Theta_n}(I_T, X) \cong
X_{k_1} \times_{X_{k'_1}} \dots \times_{X_{k'_{m-1}}} X_{k_m},
\]
where $X_l$ means $X(\Dn{l})$. We will say that $X$ \emph{satisfies the
Segal condition} if all the Segal maps are bijections. This exactly means
that $X$ is $\I$-local.
\end{paragr}

\begin{prop}[Berger, Weber]\label{prop:Theta_Segal}
The $n$-cellular nerve functor is fully faithful. Moreover, its essential
image consists of the $n$-cellular sets satisfying the Segal condition.
\end{prop}

\begin{proof}
This follows from the general machinery of \cite{Weber} (see Example 4.24).
This was first proved in \cite{BergerNerve} (see Theorem 1.12).
\end{proof}

\begin{rem}
The first assertion of the previous proposition means precisely that the
inclusion functor $\Theta_n \hookto \nCat{n}$ is dense.
\end{rem}

\subsection{The quasi-category of $\Theta_n$-spaces}

\begin{paragr}
The category of \ndef{$n$-cellular spaces} is the category $\spref{\Theta_n}
\cong \pref{\Theta_n \times \Delta}$ of simplicial presheaves on $\Theta_n$.
The first projection $p \colon \Theta_n \times \Delta \to \Theta_n$
induces a fully faithful functor $p^* \colon \pref{\Theta_n} \to \pref{\Theta_n
\times \Delta}$ sending an $n$-cellular set to the corresponding discrete
$n$-cellular space. We will always consider $n$-cellular sets as a full
subcategory of $n$-cellular spaces using this functor.
\end{paragr}

\begin{paragr}\label{paragr:def_Jk}
Let $k \ge 1$. We will denote by $J_k$ the strict \oo-category
corepresenting the functor $\Ari{k} \colon \nCat{\infty} \to \Set$ sending a strict
\oo-category to its set of strictly invertible $k$-arrows. The \oo-category
$J_k$ is actually a $k$-category. Here are pictures of (the underlying
\oo-graphs without the identities of) $J_k$ in low dimension:
\[
\UseAllTwocells
J_1 = \xymatrix{\bullet \ar@/^2ex/[r] & \ar@/^2ex/[l] \bullet}
,\quad
    J_2 = \xymatrix@C=3pc@R=3pc{\bullet \ar@/^2.5ex/[r]_{}="0"
    \ar@/_2.5ex/[r]_{}="1"
    \ar@<-1ex>@2"0";"1" \ar@<-1ex>@2"1";"0"
      &  \bullet}
\quad\text{and}\quad
    J_3 = \xymatrix@C=3pc@R=3pc{\bullet \ar@/^3ex/[r]_(.47){}="0"^(.53){}="10"
    \ar@/_3ex/[r]_(.47){}="1"^(.53){}="11"
    \ar@<2ex>@2"0";"1"_{}="2" \ar@<-2ex>@2"10";"11"^{}="3"
    \ar@<1ex>@3{->}"2";"3"_{} \ar@<1ex>@3"3";"2"_{}
    &  \bullet \pbox{.}}
\]
There is a canonical morphism $j_k \colon J_k \to \Dn{k-1}$ corepresenting
the natural transformation $\Ar_{k-1} \to \Ari{k}$ sending a $(k-1)$-arrow
to its identity. Concretely, the morphism $j_k$ is the unique morphism $J_k
\to \Dn{k-1}$ sending the two non-trivial $k$\nbd-arrows of $J_k$ to the
identity of the only non-trivial $(k-1)$-arrow of $\Dn{k-1}$.
\end{paragr}

\begin{paragr}\label{paragr:def_I_J_nCat}
Recall from paragraph~\ref{paragr:spine_nCat} that we denote by $\I$
the set
\[ \I = \{i_T\mid T \in \Ob(\Theta_n)\} \]
of spine inclusions.

We will denote by $\J^\flat$ and $\J$ the sets
\[
\J^\flat = \{j_k\mid 1 \le k \le n\}
\quad\text{and}\quad
\J = \{N_n(j_k)\mid 1 \le k \le n\}.
\]
The sets $\I$ and $\J$ will be considered as sets of maps of $n$-cellular sets
or $n$-cellular spaces depending on the context.
\end{paragr}

For the next two definitions and the proposition that follows, we will assume
that $n$ is finite.

\begin{defi}[Rezk]
The \emph{model category for $\Theta_n$-spaces} is
the left Bousfield localization of the injective model structure on
$n$-cellular spaces by the set $\I \cup \J$.
We will denote this model structure by $\spref{\Theta_n}_\ThnSp$.
\end{defi}

\begin{rem}
The original definition of $\Theta_n$-spaces in \cite{RezkThSp} proceeds by
induction on $n$. Here we are following the (trivially equivalent)
definition given in \cite[Section 7]{AraHQCat}.
\end{rem}

\begin{defi}
The \emph{quasi-category of $\Theta_n$-spaces} is the localization of the
quasi-category $\qpref{\Theta_n}$ by the set $\I \cup \J$. We will denote it
by $\ThnSp$.
\end{defi}

\begin{prop}
\label{prop:under_Theta_Sp}
The quasi-category underlying the model category of
$\Theta_n$\nbd-spaces is canonically equivalent to the quasi-category of
$\Theta_n$-spaces.
\end{prop}

\begin{proof}
We indeed have
{\allowdisplaybreaks
\begin{align*}
  \U(\spref{\Theta_n}_{\ThnSp})
& =
\U({(\I \cup \J)}^{-1}\spref{\Theta_n}_{\inj}) \\
& \cong
{(\I \cup \J)}^{-1}\U(\spref{\Theta_n}_{\inj}) \\*
& \phantom{=1} \text{(by Proposition~\ref{prop:loc_mcat_qcat})} \\
& \cong
{(\I \cup \J)}^{-1}\U(\spref{\Theta_n}_{\proj}) \\
& \phantom{=1} \text{(since $\U$ only depends on the weak equivalences)}\\
& \cong
{(\I \cup \J)}^{-1} \qpref{\Theta_n} \\*
& \phantom{=1} \text{(by Proposition~\ref{prop:desc_PA})} \\
& = \ThnSp. \qedhere
\end{align*}
}
\end{proof}

\subsection{Autoequivalences of the category of strict $n$-categories}

\begin{paragr}
Recall that for every $i$ such that $1 \le i \le n$, there exists an
autoequivalence~$\op_i$ of the category $\nCat{n}$ sending a strict
$n$-category $C$ to the strict $n$\nbd-category obtained from $C$ by
reversing the orientation of the $i$-arrows. The functor $\op_i$ is
obviously its own inverse. The mapping $i \mapsto \op_i$ extends formally to
a monoid morphism
\[ (\Z/2\Z)^n \longto \aut(\nCat{n}), \]
where for $n = \infty$, we set $(\Z/2\Z)^\infty = \prod_{k \ge 1} \Zd$.
For $\delta = (\delta_i)_{1 \le i \le n}$ an element of~$(\Z/2\Z)^n$, we
will denote by $\op_\delta$ the associated autoequivalence of $\nCat{n}$.
Concretely, if $C$ is a strict $n$-category, then $\op_\delta(C)$ is the
strict $n$-category obtained from $C$ by reversing the orientation of the
$i$-arrows for all $i$ such that $\delta_i = 1$.

Note that the monoid morphism $\Zdn \to \aut(\nCat{n})$ defines a strict
mo\-noidal functor
\[ \catZdn \longto \Aut(\nCat{n}), \]
where $\catZdn$ denotes the discrete category on the set $\Zdn$ endowed with
the strict monoidal structure given by the group law of $\Zdn$, and where
the category $\Aut(\nCat{n})$ is endowed with the strict monoidal structure
given by composition of functors.
\end{paragr}

\begin{rem}
It is easy to see that for any $\delta$ in~$\Zdn$, the autoequivalence
$\op_\delta$ sends the objects of $\Theta_n$ to $n$-categories isomorphic to
objects of $\Theta_n$.
\end{rem}

\begin{paragr}
Let $k \ge 0$. The \ndef{$k$-sphere} $\Sn{k}$ is the $k$-truncation
of the $(k+1)$\nbd-disk $\Dn{k+1}$ (say in $\nCat{\infty}$). By definition,
we have a
canonical monomorphism \hbox{$\Sn{k} \hookto \Dn{k+1}$}.
It is easy to prove by induction, setting $\Sn{-1} = \varnothing$, that
\[ \Sn{k} \cong \Dn{k} \amalg_{\Sn{k-1}} \Dn{k}. \]
\end{paragr}

\begin{lemma}\label{lemma:char_Dk}
Let $k$ be an integer such that $0 < k \le n$. The $k$-disk $\Dn{k}$ is
the unique strict $n$-category $C$ satisfying the following three properties:
\begin{enumerate}
  \item The $n$-category $C$ is not a $(k-1)$-category.
  \item Any proper sub-$n$-category of $C$ is a $(k-1)$-category.
  \item The $(k-1)$-truncation $\tr_{k-1}(C)$ of $C$ is isomorphic to $\Sn{k-1}$.
\end{enumerate}
\end{lemma}

\begin{proof}
It is obvious that $\Dn{k}$ satisfies these properties. Suppose $C$ is a
strict $n$\nbd-category satisfying these three properties. The first
property exactly means that $C$ has at least one non-trivial arrow in
dimension at least $k$. Every such non-trivial arrow of $C$ defines a
sub-$n$-category which is not a $(k-1)$-category. The second property
implies that this sub-$n$-category has to be~$C$. This means that there can
exist only one non-trivial arrow in $C$ in dimension at least $k$. Let
$l$ be the dimension of this arrow. By the third property, the
$(k-1)$\nbd-truncation of $C$ is $\Sn{k-1}$.  Since the $n$-category
generated by the unique non-trivial $l$-arrow of $C$ is $C$, the two
non-trivial $(k-1)$-arrows of $\Sn{k-1}$ have to be iterated sources or
targets of this unique non-trivial $l$-arrow. This means that $l$ has to be
equal to $k$ and that the $n$-category~$C$ has to be the $k$-disk $\Dn{k}$.
\end{proof}

\begin{lemma}\label{lemma:char_k-cat}
Let $k$ be such that $0 \le k \le n$. A strict $n$-category $C$ is a
$k$-category if and only if there exists an extremal epimorphism of the form \[
\coprod_E \Dn{k} \longto C, \]
for some set $E$, that is, if there exists an epimorphism of the above form
that does not factor through any proper subobject of $C$.
\end{lemma}

\begin{proof}
Clearly, such an extremal epimorphism exists if and only if $C$ is
generated by its set of $k$-arrows, that is, if and only if $C$ is a
$k$-category.
\end{proof}

\begin{paragr}\label{paragr:def_delta_le}
Let $F$ be an autoequivalence of $\nCat{n}$. Suppose that for some $l$ such
that $0 < l \le n$, we have
\[
  F(\Dn{l-1}) \simeq \Dn{l-1}
  \quad\text{and}\quad
  F(\Dn{l}) \simeq \Dn{l}.
\]
Since the only monomorphisms from $\Dn{l-1}$ to $\Dn{l}$ are $\sigma$ and
$\tau$, this implies that the equivalence $F$ either fixes these two
morphisms or exchanges them. We define $\delta_l(F)$ in~$\Z/2\Z$ to be $0$ if
$F$ fixes them and $1$ otherwise.

Fix now $k$ such that $0 < k \le n$ and suppose that we have $F(\Dn{l})
\simeq \Dn{l}$ for every~$l$ such that $0 \le l \le k$. We define
$\delta_{\le k}(F)$ to be the element
\[ \delta_{\le k}(F) = (\delta_1(F), \dots, \delta_k(F), 0, \dots, 0) \]
of $(\Z/2\Z)^n$. In the case $k = n$, that is, in the case where $F$
preserves all the disks, we set
\[ \delta(F) = \delta_{\le n}(F) = (\delta_1(F), \dots, \delta_n(F)). \]
\end{paragr}

\begin{prop}
Every autoequivalence $F\!$ of $\nCat{n}$ preserves the disks. In other
words, we have $F(\Dn{k}) \cong \Dn{k}$ for $0 \le k \le n$.
\end{prop}

\begin{proof}
We are going to prove by induction on $k$ that $F$ preserves both $\Dn{k}$
and~$\Sn{k-1}$. For $k = 0$, the objects $\Dn{0}$
and $\Sn{-1}$ are terminal and initial objects of $\nCat{n}$,
respectively, and hence they are preserved by $F$.

Suppose the result is true for $l < k$ and let us prove it for $k$.
The equivalence~$F$ respects monomorphisms and reflects isomorphisms. It
thus respects proper subobjects. It also reflects them since a quasi-inverse
of $F$ respects them. Moreover, since $F$ preserves sums, extremal
epimorphisms and the $(k - 1)$-disk by induction, it preserves
$(k-1)$-categories by Lemma~\ref{lemma:char_k-cat}. This shows that $F$
preserves the two first conditions of Lemma~\ref{lemma:char_Dk}.

We will now check that $F$ preserves $\Sn{k-1}$. There are exactly two
monomorphisms $\Sn{k-2} \hookto \Dn{k-1}$, say $i_1$ and $i_2$. By
induction, $F$ preserves $\Sn{k-2}$ and $\Dn{k-1}$. It thus either
fixes $i_1$ and $i_2$ or exchanges them. But for any value of $\eps$ in
$\{1, 2\}$, we have a pushout square
\[
\xymatrix{
\Sn{k-2} \ar[r]^{i_\eps} \ar[d]_{i_\eps} & \Dn{k-1} \ar[d] \\
\Dn{k-1} \ar[r] & \Sn{k-1}\pbox{.}
}
\]
Since $F$ preserves pushouts, we deduce that $\Sn{k-1}$ is preserved.

Let us now prove that $F$ is in some sense compatible with
$(k-1)$-truncation.  First, note that for any strict $n$-category $C$, the
$(k-1)$-categories $\tr_{k-1}(F(C))$ and $F(\tr_{k-1}(C))$ (recall that
we proved that $F$ preserves $(k-1)$-categories) have the same
arrows. Indeed, the set of $i$-arrows for $i < k$ is corepresented by
$\Dn{i}$, which is preserved by $F$ by induction. Since by induction $F$
preserves $\Dn{l}$ for $l$ such that $0 \le l \le k - 1$,  we have by
paragraph~\ref{paragr:def_delta_le} an element $\delta_{\le k -
1}(F)$ in $(\Z/2\Z)^n$. It is immediate that the underlying $n$-graphs of
$\tr_{k-1}(F(C))$ and $\op_{\delta_{\le k - 1}(F)}(\tr_{k-1}(C))$ are equal.

In particular, if $C$ is such that $\tr_{k-1}(C)$ is isomorphic to
$\Sn{k-1}$, we get that the underlying $n$-graph of $\tr_{k-1}(F(C))$ is the
underlying $n$-graph of
\[
  \op_{\delta_{\le k - 1}(F)}(\tr_{k-1}(C)) \cong
  \op_{\delta_{\le k - 1}(F)}(\Sn{k-1}) \cong \Sn{k-1}.
\]
But since there exists only one structure of strict $n$-category on the
underlying $n$-graph of $\Sn{k-1}$, we get that $\tr_{k-1}(F(C))$ is
isomorphic to $\Sn{k-1}$. This shows that $F$ preserves the last condition
of Lemma~\ref{lemma:char_Dk}.

We thus have proved that the three properties characterizing $\Dn{k}$ are
stable under~$F$ and hence we get that $F$ preserves $\Dn{k}$.
\end{proof}

\begin{rem}
The analogous statement for \emph{rigid} strict $n$-categories appears as
\cite[Lemma 4.5]{BarSchPrUnicity}. The proof, which is based on a different
characterization of the $n$-disk, also adapts to the case of strict
$n$-categories.
\end{rem}

\begin{paragr}\label{paragr:def_delta}\label{paragr:F_twidle_Theta}
Let $F$ be an autoequivalence of $\nCat{n}$. By the previous proposition,
$F$ preserves all the disks and we thus get by
paragraph~\ref{paragr:def_delta_le} an element $\delta(F)$ in~$\Zdn$. We
thus have a map
\[ \aut(\nCat{n}) \longto \Zdn. \]
It is immediate that this map is a monoid morphism and a retraction of
the map $\Zdn \to \aut(\nCat{n})$. In other words, if $F$ and $G$ are
autoequivalences of~$\nCat{n}$, we have
$\delta(GF) = \delta(G) + \delta(F)$, and if $\delta$ is an element of~$\Zdn$,
we have~$\delta(\op_\delta) = \delta$.

We will denote by $\wt{F}$ the autoequivalence
\[  \wt{F} = F \circ \op_{-\delta(F)}. \]
Note that we have $\delta(\wt{F}) = 0$.
\end{paragr}

\begin{prop}\label{prop:auto_nCat_restr}
If $F$ is an autoequivalence of $\nCat{n}$, then we have \hbox{$\wt{F}(T)
\cong T$} for every object $T$ of $\Theta_n$. In particular, $F$ induces an
autoequivalence of $\Theta_n$.
\end{prop}

\begin{proof}
Let $T$ be an object of~$\Theta_n$. The object $T$ seen as a strict
$n$-category is the colimit of the diagram $\D_T$ of
paragraph~\ref{paragr:def_table}. But since $\delta(\wt{F}) = 0$, the
functor $\wt{F}$ preserves the diagram $\D_T$ and since $\wt{F}$ commutes with
colimits, we indeed have~$\wt{F}(T) \cong T$.
\end{proof}

\begin{coro}\label{cor:ff_auto_nCat_Theta}
The dense inclusion $\Theta_n \hookto \nCat{n}$ induces a fully faithful
functor $\Aut(\nCat{n}) \to \Aut(\Theta_n)$.
\end{coro}

\begin{proof}
This is immediate from the previous proposition and
Proposition~\ref{prop:auto_dense}.
\end{proof}

We will show in Section \ref{subsec:autoeq_Theta} that the monoidal category
$\Aut(\Theta_n)$ is isomorphic to the discrete monoidal category $\catZdn$.
As a corollary, we will obtain the following theorem:

\begin{thm*}
The functor $\catZdn \to \Aut(\nCat{n})$ is an equivalence of monoidal
categories.
\end{thm*}

\subsection{Autoequivalences of the category of rigid strict $n$-categories}

\begin{paragr}\label{paragr:nCat_rig}
A strict $n$-category $C$ is said to be \ndef{rigid} if it contains no
non-trivial isomorphisms, that is, if for any $k$ such that $1 \le k \le n$,
any strictly invertible $k$-arrow of $C$ is the identity of a
$(k-1)$-arrow. For a fixed $k$, this condition amounts to saying that $C$
is local with respect to the map $j_k \colon J_k \to \Dn{k-1}$ of
paragraph~\ref{paragr:def_Jk}. In particular, a strict $n$-category $C$ is
rigid if and only if it is $\J^\flat$-local, where $\J^\flat$ is the set
defined in paragraph~\ref{paragr:def_I_J_nCat}.

We will denote by $\nCatr{n}$ the full subcategory of the category of strict
$n$\nbd-cate\-gories whose objects are the rigid strict $n$-categories.
\end{paragr}

\begin{paragr}
Let $\delta$ be an element of $(\Z/2\Z)^n$. It is clear that if $C$ is a
rigid strict $n$\nbd-category, then so is $\op_\delta(C)$. In other words, the
autoequivalence $\op_\delta$ induces an autoequivalence of the category
$\nCatr{n}$. We will also denote this equivalence by~$\op_\delta$. We thus
get, as in the non-rigid case, a monoidal functor
\[ \catZdn \longto \Aut(\nCatr{n}). \]

In particular, if $F$ is an autoequivalence of $\nCatr{n}$, we can define an
autoequivalence~$\wt{F}$ of $\nCatr{n}$ as in
paragraph~\ref{paragr:F_twidle_Theta}.
\end{paragr}

\begin{prop}\label{prop:autoeq_nCatr_restr_Thn}
If $F$ is an autoequivalence of $\nCatr{n}$, then $\wt{F}(T) \cong T$
for every object $T$ of $\Theta_n$. In particular, $F$ induces an
autoequivalence of $\Theta_n$.
\end{prop}

\begin{proof}
The proof used for $\nCat{n}$ in the previous subsection adapts trivially.
One only has to observe that the objects of $\Theta_n$ and the spheres are
rigid $n$\nbd-categories.
\end{proof}

\begin{coro}\label{cor:ff_auto_nCatr_Theta}
The dense inclusion $\Theta_n \hookto \nCatr{n}$ induces a fully faithful
functor $\Aut(\nCatr{n}) \to \Aut(\Theta_n)$.
\end{coro}

\begin{proof}
This is immediate from the previous proposition and
Proposition~\ref{prop:auto_dense}.
\end{proof}

We will show in Section \ref{subsec:autoeq_Theta} that the monoidal category
$\Aut(\Theta_n)$ is isomorphic to the discrete monoidal category
$\catZdn$. As a corollary, we will obtain the following theorem, first
proved by Barwick and Schommer-Pries \cite[Section 4]{BarSchPrUnicity} by
different methods:

\begin{thm*}
The functor $\catZdn \to \Aut(\nCatr{n})$ is an equivalence of monoidal
categories.
\end{thm*}

\subsection{Autoequivalences of the category
$\Theta_n$}\label{subsec:autoeq_Theta}

\begin{paragr}
Let $\delta$ be an element of $(\Z/2\Z)^n$. As already observed, if $T$ is
an object of~$\Theta_n$, then $\op_\delta(T)$ is isomorphic to an object of
$\Theta_n$. This means
that the autoequivalence $\op_\delta$ induces an autoequivalence of the category
$\Theta_n$. We will still denote this equivalence by $\op_\delta$.
We thus get, as in the previous subsections, a monoidal functor
\[ \catZdn \longto \Aut(\Theta_n). \]

In particular, if $F$ is an autoequivalence of $\Theta_n$, we can define an
autoequivalence~$\wt{F}$ of $\Theta_n$ as in
paragraph~\ref{paragr:F_twidle_Theta}.
\end{paragr}

\begin{lemma}
Let $k$ be such that $0 < k \le n$. The $k$-disk $\Dn{k}$ is the unique
object~$T$ of $\Theta_n$ satisfying the following two properties:
\begin{enumerate}
  \item Every proper subobject of $T$ is an $l$-disk for $l < k$.
  \item For every $l$ such that $0 \le l < k$, there are exactly two
    monomorphisms from~$\Dn{l}$ to $T$.
\end{enumerate}
\end{lemma}

\begin{proof}
Let
\[
T = \tabdimk
\]
be an object of $\Theta_n$ satisfying the two properties of the statement. First
note that we must have $m \le 2$ for otherwise
\[
\left(
\begin{matrix}
  k_1 && k_2 && \cdots && k_{m-1} \cr
  & k'_1 && k'_2 & \cdots & k'_{m-2}
\end{matrix}
\right)
\]
would be a proper subobject that is not a disk, contradicting property
(i). Suppose $m = 2$ so that
\[
  T =
\left(
\begin{matrix}
  k_1 && k_2  \cr
  & k'_1
\end{matrix}
\right).
\]
There are exactly three monomorphisms from $\Dn{k_1'}$ to such a $T$.
Property (i) implies that $k_1' < k$ and we get a contradiction with
property (ii). This means that $m = 1$ or, in other words, that $T$ is a
disk $\Dn{p}$. Property (i) implies that $p \le k$ and property~(ii) that $p
\ge k$, thereby proving the result.
\end{proof}

\begin{prop}\label{prop:auto_Thn}
If $F$ is an autoequivalence of $\Theta_n$, then $\wt{F}$ is the identity on
objects.
\end{prop}

\begin{proof}
The strategy is similar to the one used for $\nCat{n}$ but has to be adapted
since the spheres are not objects of $\Theta_n$. Nevertheless, for $k$ such
that $0 < k \le n$, the characterization of the $k$-disk given by
Lemma~\ref{lemma:char_Dk} can be replaced by the one given by the above
lemma. Using this characterization, we get by induction on $k$  that $\Dn{k}$ is
preserved (starting from $\Dn{0}$ which is the terminal object of
$\Theta_n$). We then obtain that every object is preserved by expressing an
object $T$ of $\Theta_n$ as the colimit of the diagram $\D_T$ (using the
same argument as in the proof of Proposition~\ref{prop:auto_nCat_restr}).
\end{proof}

\begin{prop}\label{prop:auto_Theta_n}
The monoid morphism $\Zdn \to \aut(\Theta_n)$ is an isomorphism.
\end{prop}

\begin{proof}
It suffices to show that the retraction $F \mapsto \delta(F)$ (see
paragraph~\ref{paragr:def_delta}) of the morphism of the statement is
injective, that is, that if $F$ is an autoequivalence of $\Theta_n$ such
that $\delta(F) = 0$, then $F$ is the identity.

Let us fix such an $F$. We know that $F$ is the identity on objects by the
previous proposition. Since morphisms of strict $n$-categories are
determined by their action on arrows, it suffices to show that morphisms of
the form $\Dn{k} \to T$, where $k$ is such that $0 \le k \le n$ and $T$ is
any object of $\Theta_n$, are preserved.

Let $\Dn{k} \to T$ be such a morphism. By
Proposition~\ref{prop:active_inert}, this morphism factors as a composite
\[ \Dn{k} \stackrel{a}{\longto} S \stackrel{i}{\longto} T, \]
where $a$ is active and $i$ is inert.
Let us first prove that $i$ is preserved by $F$. Recall that $S$ and $T$ are
colimits of diagrams $\D_S$ and $\D_T$ involving disks only. Consider the
cocone $\D_S \to T$ associated to~$i$. Since $i$ is inert, each of the
components $\Dn{l} \to T$ of this cocone corresponds to an $l$-cell of the
$n$-graph $\Theta_0(T)$ generating $\Theta(T)$ and can thus be written (in a
non-canonical way) as a composite
\[ \Dn{l} \longto \Dn{l'} \longto \colim \D_T = T, \]
where the first map is either $\sigma$ or $\tau$, and the second map is one
of the canonical morphisms associated to $\D_T$. Since $\delta(F) = 0$, the
morphism $\Dn{l} \to \Dn{l'}$, the diagram~$\D_S$ and the diagram $\D_T$ are
preserved by $F$. Since moreover, as already observed, the cocone making $T$
a colimit of the diagram $\D_T$ is unique, the canonical morphism $\Dn{l'}
\to T$ is also preserved. This shows that $i$ is preserved.

Finally, let us prove that $a \colon \Dn{k} \to S$ is preserved by $F$. Since
there is at most one active morphism from a fixed disk to a fixed object of
$\Theta_n$ (see paragraph~\ref{paragr:active_inert}), it suffices to show
that $F(a)$ is active. This follows from the fact, just proved, that any
autoequivalence of $\Theta_n$ preserves inert morphisms. Indeed, by
Proposition~\ref{prop:active_inert}, $F(a)$ factors uniquely as $ib$, where
$i$ is inert and $b$ is active. This means that $a = F^{-1}(i)F^{-1}(b)$. By
decomposing $F^{-1}(b)$, we get that $a = F^{-1}(i)jc$, where $j$ is inert
and $c$ is active. Since $F^{-1}$ preserves inert morphisms, $F^{-1}(i)$ is
also inert. By uniqueness of the decomposition of $a$, we obtain that
$F^{-1}(i)j$ is an identity. Since inert morphisms are monomorphisms and
$\Theta_n$ has no non-trivial isomorphisms, this implies that
$F^{-1}(i)$ and hence $i$ are identities, and therefore that $F(a)$ is
active.
\end{proof}

\begin{thm}\label{thm:auto_Thn}
The functor $\catZdn \to \Aut(\Theta_n)$ is an isomorphism of mo\-noidal
categories.
\end{thm}

\begin{proof}
The previous proposition states that this functor is bijective on objects.
To conclude, it suffices to show that $\Aut(\Theta_n)$ is a discrete
category. Let $\delta$ and $\delta'$ be two elements of $(\Z/2\Z)^n$, and
let $\gamma \colon \op_\delta \to \op_{\delta'}$ be a natural transformation.
For every object $T$ of $\Theta_n$, we thus have a morphism
$\gamma^{}_T : \op_{\delta}(T) \to \op_{\delta'}(T)$ and, in particular, for
every $0 \le k \le n$, we have an endomorphism $\gamma^{}_{\Dn{k}}$ of
$\Dn{k}$.

Let us show by induction that for every $0 \le k \le n$, we have
$\delta^{}_k = \delta'_k$ and $\gamma^{}_{\Dn{k}} = \id{\Dn{k}}$. (A priori,
$\delta^{}_0$ and $\delta'^{}_0$ are not defined and we define them to be
both equal to $0$ for the purpose of starting our induction.)
The case $k = 0$ is obvious.  For $k \ge 1$, consider the naturality squares associated to $\sigma, \tau \colon
\Dn{k-1} \to \Dn{k}$:
\[
\xymatrix{
\Dn{k - 1} \ar[r]^{\op_\delta(\sigma)} \ar[d]_{\gamma^{}_{\Dn{k-1}} = \id{\Dn{k-1}}} & \Dn{k} \ar[d]^{\gamma^{}_{\Dn{k}}} \\
\Dn{k - 1} \ar[r]_{\op_{\delta'}(\sigma)} & \Dn{k} \pbox{,}
}
\quad
\xymatrix{
\Dn{k - 1} \ar[r]^{\op_\delta(\tau)} \ar[d]_{\gamma^{}_{\Dn{k-1}} =
\id{\Dn{k-1}}} & \Dn{k} \ar[d]^{\gamma^{}_{\Dn{k}}} \\
\Dn{k - 1} \ar[r]_{\op_{\delta'}(\tau)} & \Dn{k} \pbox{.}
}
\]
These squares determine the value of $\gamma^{}_{\Dn{k}}$ on the
$(k-1)$-truncation $\Sn{k - 1}$ of $\Dn{k}$: it is one of the two
automorphisms of $\Sn{k - 1}$. But only the trivial automorphism can be
lifted to an endomorphism of $\Dn{k}$ and the unique lift is then the
identity of $\Dn{k}$. This exactly means that $\delta^{}_k$ and $\delta'_k$
are equal and that $\gamma^{}_{\Dn{k}}$ is the identity of $\Dn{k}$.

This shows that if such a $\gamma$ exists, then $\delta = \delta'$. To
conclude, we have to show that the identity is the unique natural
transformation $\gamma \colon \op_\delta \to \op_\delta$. Fix an object~$T$
of~$\Theta_n$ and consider the naturality squares
\[
\xymatrix@C=3pc{
\Dn{k} \ar[r]^-{\op_\delta(u)} \ar[d]_{\id{\Dn{k}}} & \op_\delta(T) \ar[d]^{\gamma^{}_T}
\\
\Dn{k} \ar[r]_-{\op_{\delta}(u)} & \op_\delta(T)
}
\]
associated to morphisms of the form $u \colon \Dn{k} \to T$ where $0 \le k \le
n$.
Since a morphism of strict $n$-categories is determined by its action on
arrows, there is at most one morphism $\gamma^{}_T$ making all these squares
commute, namely the identity of~$T$, thereby proving the result.
\end{proof}

\begin{thm}
The monoidal categories $\Aut(\nCat{n})$ and $\Aut(\nCatr{n})$ are
equivalent to the discrete monoidal category $\catZdn$.
\end{thm}

\begin{proof}
By Corollaries~\ref{cor:ff_auto_nCat_Theta}
and~\ref{cor:ff_auto_nCatr_Theta}, we have that the categories $\Aut(\nCat{n})$ and
$\Aut(\nCatr{n})$ are both full (monoidal) subcategories of the category
$\Aut(\Theta_n)$. To conclude, it thus suffices to show that every
autoequivalence of $\Theta_n$ lifts to auto\-equivalences of $\nCat{n}$ and
$\nCatr{n}$. This is obvious since, by Proposition~\ref{prop:auto_Theta_n},
the autoequivalences of $\Theta_n$ are the $\op_\delta$.
\end{proof}

\begin{rem}
\tolerance=500
The fact that the category $\Aut(\nCatr{n})$ is equivalent to the discrete
category $\catZdn$ was first obtained by Barwick and Schommer-Pries
in~\cite[Section 4]{BarSchPrUnicity}.
\end{rem}

\subsection{Autoequivalences of the $(\infty, 1)$-category of $(\infty,
n)$-categories}

\begin{paragr*}
In this section, we suppose that $n$ is finite.
\end{paragr*}

\begin{paragr}\label{paragr:op_ThnSp}
Let $\delta$ be an element of $\Zdn$. The autoequivalence $\op_\delta$ of
$\Theta_n$ extends formally to an autoequivalence of the quasi-category
$\qpref{\Theta_n}$. It is easy to see that the sets $\I$ and $\J$ of
paragraph~\ref{paragr:def_I_J_nCat} are stable under this autoequivalence
and we thus get an induced autoequivalence $\op_\delta$ of $\ThnSp = {(\I \cup
\J)}^{-1}\qpref{\Theta_n}$.
\end{paragr}

\begin{prop}\label{prop:desc_local_nCat}
An $n$-cellular set is $(\I \cup \J)$-local if and only if it is
the nerve of a rigid strict $n$-category.
\end{prop}

\begin{proof}
Proposition~\ref{prop:Theta_Segal} precisely says that an $n$-cellular set
is $\I$-local if and only if it is the nerve of a strict $n$-category. By
paragraph~\ref{paragr:nCat_rig}, such an $n$-cellular set is $\J$-local if
and only if the strict $n$-category of which it is the nerve is rigid,
thereby proving the result.
\end{proof}

\begin{thm}\label{thm:thn}
The quasi-category $\Aut(\ThnSp)$ is canonically equivalent to the
discrete category $\catZdn$.
\end{thm}

\begin{proof}
We are going to apply Proposition~\ref{prop:lemma_aut} to $A = \Theta_n$ and
$S = \I \cup \J$. Let us check that the hypotheses are fulfilled. Using the
previous proposition, this amounts to verifying that
\begin{enumerate}
  \item objects of $\Theta_n$ are rigid strict $n$-categories;
  \item autoequivalences of $\nCatr{n}$ restrict to autoequivalences of
    $\Theta_n$.
\end{enumerate}
The first point is obvious and the second point is
Proposition~\ref{prop:autoeq_nCatr_restr_Thn}. We can thus apply the
proposition and we get that $\Aut(\ThnSp)$ is a full subcategory
of $\Aut(\Theta_n)$. But $\Aut(\Theta_n)$ is isomorphic to $\catZdn$ by
Theorem~\ref{thm:auto_Thn}. To conclude, it thus suffices to show that every
autoequivalence of $\Theta_n$ lifts to an autoequivalence of~$\ThnSp$. This
follows from paragraph~\ref{paragr:op_ThnSp}.
\end{proof}

\begin{rem}
The above result is a consequence of previous work of Barwick and
Schommer-Pries.  Indeed, by \cite[Theorem 11.15]{BarSchPrUnicity} the
quasi-category $\ThnSp$ is equivalent to the quasi-category of
$\Upsilon_n$\nbd-spaces, and by \cite[Theorem 8.12]{BarSchPrUnicity} the
quasi-category of autoequivalences of $\Upsilon_n$-spaces is~$\Zdn$.
\end{rem}

\begin{rem}\label{rem:oomon}
Let $\C$ be a quasi-category and let $\Fun(\C,\C)$ be the associated
quasi-category of endofunctors. Recall from
\cite[Proposition~3.1.7]{lurie:DAGII} that $\Fun(\C,\C)$ is a monoidal
quasi-category, the monoidal structure being given by the composition of
endofunctors (for the general theory of \emph{non-symmetric}
monoidal quasi-categories, see \cite{lurie:DAGII}). Moreover, the full
subcategory $\Aut(\C)\subseteq\Fun(\C,\C)$ spanned by the autoequivalences
inherits a monoidal structure.

Using these monoidal structures, Theorem~\ref{thm:thn} admits the following
refinement: the
quasi-categories $\Aut(\ThnSp)$ and $\catZdn$ are equivalent as monoidal quasi-categories.
\end{rem}

\section{Autoequivalences of the \pdfinftyo-category of \pdfinfty-operads}
\label{sec:operads}

In this section, we show that the quasi-category of autoequivalences of the
quasi-category of $\Omega$-spaces, which is a model for \oo-operads, is a
contractible Kan complex. We follow the general strategy described in the
introduction and formalized by Proposition~\ref{prop:lemma_aut}. In
particular, we reduce to the computation of the autoequivalences of the
category $\Omega$ of trees.

\subsection{Preliminaries on operads and the category of trees}
\label{subsec:trees}

\begin{paragr}
We will denote by $\Oper$ the category of (small) symmetric coloured
operads (see for instance \cite[Section 2]{ElmMand} where they are called
multicategories). We will refer to its objects simply as operads.

If $P$ is an operad and $c_1, \dots, c_n$ and~$d$ are colours of $P$, we
will denote by $P(c_1, \dots, c_n; d)$ the set of operations of $P$ whose
input colours are given by the $n$-tuple $(c_1, \dots, c_n)$ and whose output
colour is~$d$. If $p \in P(c_1,\ldots,c_n;d)$ is an $n$-ary operation and
$\sigma$ is an element of the symmetric group $\Sigma_n$, we will denote by
$p\sigma$ the induced operation in
$P(c_{\sigma(1)},\ldots,c_{\sigma(n)};d)$.

We will denote by $\Coll$ the category of symmetric (coloured) collections
and by~$\nsColl$ the category of non-symmetric (coloured) collections. Recall that
a symmetric collection $K$ consists of a set of \ndef{colours} and, for
every $(n+1)$-tuple of colours $(c_1,\ldots, c_n,d)$ with $n\ge 0$, a
set $K(c_1,\dots, c_n; d)$ with an action of~$\Sigma_n$. Non-symmetric
collections are described in the same way, but forgetting the action of the
symmetric groups. We have forgetful functors
\[ \Oper \longto \Coll \longto \nsColl \]
and these functors admit left adjoints.

If $p$ and $q$ are two operations of an operad $P$, we will write $p
\sim_\Sigma q$ if $p$ and $q$ have the same arity $n$ and there exists an
element $\sigma$ in $\Sigma_n$ such that $q = p\sigma$.

We will say that an operation of an operad is \emph{non-trivial} if it is
not the identity of a colour.

We will identify the category $\Cat$ of small categories with the full
subcategory of $\Oper$ consisting of operads having only unary operations.
The inclusion functor $\Cat \hookrightarrow \Oper$ admits a right adjoint
sending an operad to its so-called \emph{underlying category}. Concretely,
the underlying category of an operad $P$ is the suboperad of~$P$ obtained by
throwing out the non-unary operations (in particular, its objects are the
colours of $P$).
\end{paragr}

\begin{paragr}
A \emph{combinatorial tree} is a non-empty finite connected graph with no
loops. A~vertex of a combinatorial tree is said to be \emph{outer} if it has
only one edge attached to it.

An \emph{operadic tree} is a combinatorial tree endowed with the choice of
an outer vertex called the \emph{output} and of a (possibly empty) set of
outer vertices not containing the output called the set of \emph{inputs}.
The choice of the output induces an orientation of the tree ``from the
inputs to the output''.
A \emph{vertex} of an operadic tree is defined as a vertex of the underlying
combinatorial tree which is neither an input nor the output. An \emph{edge}
of an operadic tree is an edge of the underlying combinatorial tree. The
edge attached to the output will be called the \emph{root} and the
edges attached to the inputs will be called the \emph{leaves}.

We will follow the usual conventions when drawing operadic trees: the output
will be drawn at the bottom of the tree and the vertices corresponding to
the inputs and the output will be deleted. Here is an example of such a
tree:

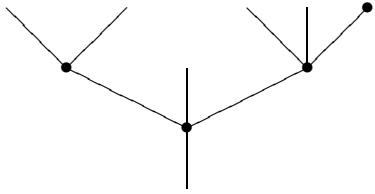
\begin{equation*}
\xy<0.08cm, 0cm>: (0,-15)*{}="1"; (0,-5)*=0{\bullet}="2";
(-20,5)*=0{\bullet}="3"; (0,5)*=0{}="4"; (20,5)*=0{\bullet}="5";
(-30,15)*{}="6"; (-10,15)*{}="7"; (10,15)*{}="8"; (20,15)*{}="9";
(30,15)*=0{\bullet}="10";
(-10,-15); "1";"2" **\dir{-}; "2";"3" **\dir{-}; "2";"4"
**\dir{-}; "2";"5" **\dir{-}; "3";"6" **\dir{-}; "3";"7" **\dir{-};
"5";"8" **\dir{-}; "5";"9" **\dir{-}; "5";"10" **\dir{-};
\endxy
\end{equation*}
From now on, by a \emph{tree}, we will mean an operadic tree.

A \emph{planar structure} on a tree $T$ consists of the data of an ordering
of the input edges of each vertex $v$ of $T$. A \emph{planar tree}
is a tree endowed with a planar structure. We will sometimes use
\emph{non-planar tree} as a synonym of tree in order to emphasize that a
given tree is not planar.
\end{paragr}

\begin{paragr}
\label{paragr:assoc_operad}
To every tree $T$, we can associate a symmetric collection $K(T)$. The
colours of $K(T)$ are the edges of $T$ and for every vertex $v$ of $T$ with
$n$ inputs and every choice of an order $e_1 < \dots < e_n$ on the input edges
of $v$, there is an operation in $K(T)(e_1, \dots, e_n; e)$, where $e$ is
the output edge of $v$. The group $\Sigma_n$ acts on the operations
associated to $v$ in the obvious way.

If $T$ is a tree, we will denote by $\Omega(T)$ the free symmetric operad on
the symmetric collection $K(T)$.

Similarly, to every planar tree $\overline{T}$, we can associate a
non-symmetric collection~$K_p(\overline{T})$. Its colours are the edges of
$\overline{T}$ and for every vertex $v$ in $\overline{T}$ with input edges
$e_1 < \dots < e_n$ and output edge $e$, there is an operation in
$K_p(\overline{T})(e_1, \dots, e_n; e)$.

Note that if $T$ is the underlying non-planar tree of a planar tree
$\overline{T}$, then $K(T)$ is the free symmetric collection on
$K_p(\overline{T})$. (Another way to put this is to say that the choice of a
planar structure on $T$ corresponds to a choice of generators of $K(T)$.) In
particular, $\Omega(T)$ is the free symmetric operad on $K_p(\overline{T})$.
\end{paragr}

\begin{paragr}\label{paragr:omega}
The category of trees $\Omega$, introduced by Moerdijk and Weiss in
\cite{MoerWeissDend}, is defined as follows: the objects of
$\Omega$ are trees up to isomorphism (that is, up to renaming of their
vertices and edges) and the set of morphisms in $\Omega$ from an
object~$S$ to an object $T$ is given by
\[
\Omega(S, T) = \Oper(\Omega(S), \Omega(T)).
\]
By definition, there is a canonical fully faithful functor $\Omega\hookto\Oper$
and we will always consider $\Omega$ as a full subcategory of $\Oper$ using
this functor.

We will denote by $\eta$ the tree with one edge and no vertices. For $n \ge
0$, we will denote by $C_n$ the $n$-corolla, that is, the tree with one
vertex and $n$ leaves. Note that $C_n$, seen as an object of $\Oper$,
corepresents the functor ``set of $n$-ary operations''.  Similarly, $\eta$
corepresents the functor ``set of colours''. In particular, for any tree
$T$, we have a \emph{root map} $\eta \to T$ and \emph{leaf maps} $\eta\to
T$.
\end{paragr}

\emph{We now fix for every object $T$ of $\Omega$ the choice of a planar
structure on $T$.} We will denote the resulting planar tree by
$\overline{T}$. The purpose of these choices is to make precise the idea
that every tree can be obtained by glueing corollas (see
Proposition~\ref{prop:can_decomp}).

\begin{paragr}
\label{paragr:omega_0}
Given a tree $T$, we will denote by $\Omega_0(T)$ the non-symmetric
collection $K_p(\overline{T})$. We define a category $\Omega_0$ in the
following way: the objects of $\Omega_0$ are the same as the ones of
$\Omega$ and the set of morphisms in $\Omega_0$ from an object $S$ to an
object $T$ is given by
\[
\Omega_0(S, T) = \nsColl(\Omega_0(S), \Omega_0(T)).
\]
The free symmetric operad functor on a non-symmetric collection induces a
canonical functor from $\Omega_0$ to $\Omega$ which is obviously faithful.
We will always consider $\Omega_0$ as a subcategory of $\Omega$ using this
functor.
\end{paragr}

\begin{rem}\label{rem:Omega_bar}
The subcategory $\Omega_0$ of $\Omega$ \emph{does} depend on the choice of the
planar structures. One way to avoid these choices is to replace $\Omega$ by
the following equivalent category $\overline{\Omega}$: an object of
$\overline{\Omega}$ is an object of $\Omega$ endowed with the choice of a
planar structure and the set of morphisms in $\overline{\Omega}$ from an
object $\overline{S}$ to an object $\overline{T}$ is given by
\[
\overline{\Omega}(\overline{S}, \overline{T}) =
\Oper(\Omega(S), \Omega(T)),
\]
where $S$ and $T$ denote the respective underlying non-planar trees of
$\overline{S}$ and $\overline{T}$. It is then possible to define a canonical
subcategory $\overline{\Omega}_0$
of $\overline{\Omega}$ by setting
\[ \overline{\Omega}_0(\overline{S},\overline{T}) =
\nsColl(K_p(\overline{S}), K_p(\overline{T})). \]

Note that the choice of planar structures we made corresponds to the choice of
a section of the equivalence of categories given by the forgetful functor
$\overline{\Omega}\to\Omega$.
\end{rem}

\begin{paragr}\label{paragr:can_decomp}
We will denote by $\Cor$ the full subcategory of $\Omega_0$ whose objects are
$\eta$ and the corollas. The morphisms of $\Cor$, besides the identities,
are exactly the root and the leaf maps of corollas. We will denote by $j$
the inclusion functor $\Cor \hookrightarrow \Omega_0$.

The category of non-symmetric collections can be identified with the
category of presheaves on $\Cor$. Indeed, a presheaf $F$ on $\Cor$ is given
by a set $F(\eta)$ (corresponding to the set of colours) and, for every
$n \ge 0$, a set $F(C_n)$ endowed with a morphism $F(C_n) \to F(\eta)^{n+1}$
induced by the $n$ leaf maps and the root map $\eta \to C_n$. But such a map
amounts to a family of sets indexed by $(n+1)$-tuples $(c_1, \dots, c_n, d)$
of elements of $F(\eta)$.

If $T$ is an object of $\Omega$, we will denote by $\Cor/T$ the comma
category $j \downarrow T$, where $T$ is seen as an object of $\Omega_0$. An
object of $\Cor/T$ is hence a pair $(C, C \to T)$, where $C$ is an object of
$\Cor$ and $C \to T$ is a morphism of $\Omega_0$. A morphism from $(C, C \to
T)$ to $(C', C' \to T)$ is a morphism $C \to C'$ in $\Cor$ (i.e., in
$\Omega_0$) making the obvious triangle commute. Thus, an object of $\Cor/T$ is
just a colour or a generating operation of~$T.$

We will denote by $\DT_T$ the functor $\Cor/T \to \Cor \hookrightarrow
\Omega$.
\end{paragr}

\begin{prop}\label{prop:can_decomp}
For every object $T\!$ of $\,\Omega$, the canonical morphism
\[ \colim \DT_T = \colim_{(C, C \to T) \in \Cor/T}\limits\!\!\!\!\! C \,
\longrightarrow\, T \]
is an isomorphism of $\Omega$. Moreover, the inclusion functor $\Omega
\hookrightarrow \Oper$ preserves this colimit.
\end{prop}

\begin{proof}
Since $\Omega$ is a full subcategory of $\Oper$, it suffices to prove that
we have a canonical isomorphism when the colimit is taken in $\Oper$. The
canonical decomposition of the non-symmetric collection $K_p(\overline{T})$, seen as a
presheaf over $\Cor$, as a colimit of representable presheaves gives a
canonical isomorphism of non-symmetric collections
\[ \colim_{(C, C \to T) \in \Cor/T}\limits \!\!\!\!\!K_p(C) \,\longrightarrow\, K_p(\overline{T}). \]
The result then follows from the fact that the free symmetric operad functor
commutes with colimits.
\end{proof}

\begin{exam}
Let $T$ be the tree
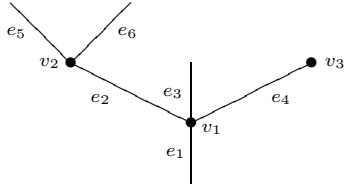
\[
\xy<0.08cm, 0cm>: (0,-15)*{}="1"; (0,-5)*=0{\bullet}="2";
(-20,5)*=0{\bullet}="3"; (0,5)*=0{}="4"; (20,5)*=0{\bullet}="5";
(-30,15)*{}="6"; (-10,15)*{}="7";
(-2.5,-10)*=0{\scriptstyle e_1};
(3.5,-6)*=0{\scriptstyle v_1}; (-3,0)*=0{\scriptstyle e_{3}};
(-23.5,5)*=0{\scriptstyle v_2}; (24,5)*=0{\scriptstyle v_3};
(-29,10)*=0{\scriptstyle e_{5}};
(-10.5,10)*=0{\scriptstyle e_{6}}; (15,-1)*=0{\scriptstyle e_{4}};
(-15,-1)*=0{\scriptstyle e_2};(-10,-15)*{};
"1";"2" **\dir{-}; "2";"3" **\dir{-}; "2";"4"
**\dir{-}; "2";"5" **\dir{-}; "3";"6" **\dir{-}; "3";"7" **\dir{-};
\endxy
\]
endowed with the planar structure given by the picture.
We define $T_i$ as the corolla associated to the vertex $v_i$, i.e., $T_i$ has
the input edges of $v_i$ as leaves and the output edge of $v_i$ as root.
Then, the diagram $\DT_T$ has the corollas $T_i$ and a copy~$\eta_e$ of
$\eta$ for every edge $e$ in $T$ as objects. The morphisms in the diagram
are the morphisms from each $\eta_e$ to the corresponding edge
of~$T_i$. Thus, the diagram $\DT_T$ is the following:
\smallskip
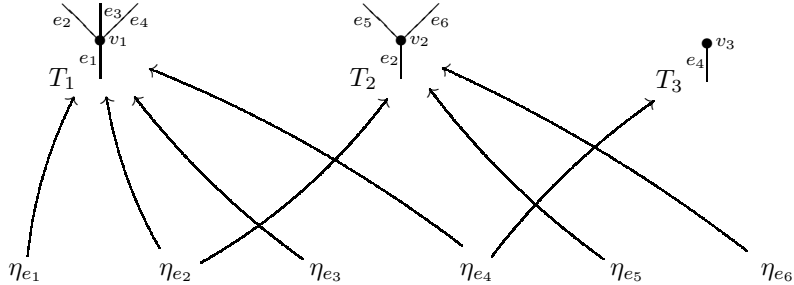
\[
\xy<0.1cm,0cm>:
(-30,0)*{
\xy<0.05cm,0cm>:
(0,20)*=0{}="1";
(20,20)*=0{}="2";
(10,10)*=0{\bullet}="3";
(10,20)*=0{}="35";
(10,0)*=0{}="4";
(0,15)*{\scriptstyle e_{2}};
(20.5,15)*{\scriptstyle e_{4}};
(15,10)*{\scriptstyle v_1};
(13.5,17.5)*{\scriptstyle e_{3}};
(7,5)*{\scriptstyle e_1};
(0,0)*=0{T_1};
(0,-5)*=0{};
"1";"3" **\dir{-};
"2";"3" **\dir{-};
"3";"4" **\dir{-};
"3";"35" **\dir{-};
\endxy
}="T1";
(10,0)*-{
\xy<0.05cm,0cm>:
(0,20)*=0{}="1";
(20,20)*=0{}="2";
(10,10)*=0{\bullet}="3";
(10,0)*=0{}="4";
(0,15)*{\scriptstyle e_{5}};
(20.5,15)*{\scriptstyle e_{6}};
(15,10)*{\scriptstyle v_2};
(7,5)*{\scriptstyle e_2};
(0,0)*=0{T_2};
(0,-5)*=0{};
"1";"3" **\dir{-};
"2";"3" **\dir{-};
"3";"4" **\dir{-};
\endxy
}="T2";
(50,-2.5)*++{
\xy<0.05cm,0cm>:
(10,10)*=0{\bullet}="3";
(10,0)*=0{}="4";
(15,10)*{\scriptstyle v_3};
(7,5)*{\scriptstyle e_4};
(0,0)*=0{T_3};
(0,-5)*=0{};
"3";"4" **\dir{-};
\endxy
}="T3";
(-40,-30)*++{
\xy<0.05cm,0cm>:
(-20,0)*=0{\eta_{e_1}};
\endxy
}="eta1";
(-20,-30)*+++{
\xy<0.05cm,0cm>:
(-20,0)*=0{\eta_{e_2}};
\endxy
}="eta2";
(0,-30)*+++{
\xy<0.05cm,0cm>:
(-20,0)*=0{\eta_{e_3}};
\endxy
}="eta3";
(20,-30)*++{
\xy<0.05cm,0cm>:
(-20,0)*=0{\eta_{e_4}};
\endxy
}="eta4";
(40,-30)*+++{
\xy<0.05cm,0cm>:
(-20,0)*=0{\eta_{e_5}};
\endxy
}="eta5";
(60,-30)*++++{
\xy<0.05cm,0cm>:
(-20,0)*=0{\eta_{e_6}}="e6";
\endxy
}="eta6";
\ar@{->}@/_/"T2";"eta6";
\ar@{->}@/^/"T3";"eta4";
\ar@{->}@/^/"T2";"eta5";
\ar@{->}@/_/(8,-6.5);"eta2";
\ar@{->}@/^/"T1";"eta1";
\ar@{->}@/^/"T1";"eta2";
\ar@{->}@/^/"T1";"eta3";
\ar@{->}@/_/"T1";(18,-26);
\endxy
\]
\end{exam}

\begin{paragr}
The category of \ndef{dendroidal sets} is the category $\pref{\Omega}$ of
presheaves on $\Omega$. The inclusion $\Omega \hookto \Oper$ induces a
functor $N_d \colon \Op \to \Dend$ sending an operad~$P$ to the dendroidal set $T
\mapsto \Op(T, P)$.  This functor $N_d$ is called the \emph{dendroidal
nerve}.
\end{paragr}

\begin{paragr}\label{paragr:dend_spines}
Let $T$ be a tree. The \emph{spine} of $T$ is the dendroidal set
\[
  I_T = \colim_{(C, C \to T) \in \Cor/T}\limits\!\!\!\!\! C,
\]
where the colimit is taken in $\Dend$. By
Proposition~\ref{prop:can_decomp}, there is a canonical morphism of
dendroidal sets $i_T \colon I_T \to T$. It is not hard to check that this
morphism is a monomorphism. We will denote by~$\I$ the set
\[ \I = \{i_T\mid T \in \Ob(\Omega)\} \]
of spine inclusions.

Let now $X$ be a dendroidal set. For any tree $T$, the map $i_T$ induces a
\emph{Segal map}
\[
X(T) \cong \Dend(T, X) \longrightarrow \Dend(I_T, X) \cong
\lim_{(C, C \to T) \in \Cor/T}\limits\!\!\!\!\! X(C).
\]
We will say that $X$ \emph{satisfies the Segal condition} if all the Segal
maps are bijections. This exactly means that $X$ is $\I$-local.
\end{paragr}

\begin{prop}[Cisinski--Moerdijk, Weber]\label{prop:dend_Segal}
The dendroidal nerve functor is fully faithful. Moreover, its essential
image consists of the dendroidal sets satisfying the Segal condition.
\end{prop}

\begin{proof}
This is \cite[Corollary 2.6]{CisMoerdDendSegal}. It also follows from the
general machinery of \cite{Weber} (see Example 4.27).
\end{proof}

\begin{rem}
The first assertion of the previous proposition precisely means that the
inclusion functor $\Omega \hookto \Oper$ is dense.
\end{rem}

\subsection{The quasi-category of complete dendroidal Segal spaces}

\begin{paragr}\label{paragr:dend_spaces}
The category of \ndef{dendroidal spaces} is the category $\spref{\Omega}
\cong \pref{\Omega \times \Delta}$ of simplicial presheaves on $\Omega$.
The first projection $p \colon \Omega \times \Delta \to \Omega$ induces a fully
faithful functor $p^* \colon \pref{\Omega} \to \pref{\Omega \times \Delta}$
sending a dendroidal set to the corresponding discrete dendroidal space. We
will always consider dendroidal sets as a full subcategory of dendroidal
spaces using this functor.
\end{paragr}

\begin{paragr}
Let $X$ be a dendroidal set and let $T$ be a tree. Then, by functoriality,
the group $\aut_\Omega(T)$ acts on $X(T)$. More generally, if $f \colon X \to Y$
is a monomorphism of dendroidal sets, then $\aut_\Omega(T)$ acts on
the difference $Y(T) \bs f_T(X(T))$. Such a monomorphism is said to be
\emph{normal} if this action is free for any $T$.

More generally, a monomorphism of dendroidal spaces $f \colon X \to Y$ will be
called \emph{normal} if the monomorphism of
dendroidal sets $f_{\bullet, n} \colon X_{\bullet, n} \to Y_{\bullet, n}$ is
normal for every $n \ge 0$.
\end{paragr}

\begin{prop}[Cisinski--Moerdijk]
There is a simplicial proper combinatorial model structure on
$\spref{\Omega}$ whose weak equivalences are the objectwise simplicial weak
homotopy equivalences and whose cofibrations are the normal monomorphisms.
\end{prop}

\begin{proof}
The existence of this model structure is given by \cite[Proposition
5.2]{CisMoerdDendSegal}, as well as the fact that it is proper and
combinatorial.

Let us prove it is simplicial. We have to show that the
pushout product $f \boxp g$ of a monomorphism $f$ of
simplicial sets and a normal monomorphisms $g$ of dendroidal spaces is again
a normal monomorphisms of dendroidal sets, and that moreover, if either $f$
or $g$ is a weak equivalence, so is $f \boxp g$. Using the fact that
pushouts in a presheaf category are computed objectwise, the first statement amounts
to showing that the pushout product of a map of sets and a normal
monomorphism of dendroidal sets is again a normal monomorphism of dendroidal
sets. This follows at once by inspection. The second statement follows
immediately from the fact that the Kan--Quillen model structure is cartesian
closed.
\end{proof}

Following \cite{CisMoerdDendSegal} and \cite{BergMoerdGenReedy}, the model
structure of the previous proposition will be called the \emph{generalized
Reedy model structure} on dendroidal spaces. We will denote it by
$\spref{\Omega}_{\textrm{Reedy}}$.

\begin{paragr}\label{paragr:BV}
Recall that the category of operads is endowed with a closed symmetric monoidal
structure given by the so-called Boardman--Vogt tensor product
\cite[Definition 2.14]{BoardVogt} (see also \cite[Part I, Section
4.1]{MoerdToen}). We will denote this tensor product by $\tensOp$. The
associated internal hom will be denoted by $\HomiOp$. The unit of
this tensor product is the operad $\eta$. In particular, for any operad $P$,
we have a canonical isomorphism $\HomiOp(\eta, P) \cong P$.

The only thing we will need to know about $\HomiOp$ is the description of
its underlying category: if $P$ and $Q$ are two operads, then the set of
objects of the underlying category of $\HomiOp(P, Q)$ is the set of maps of
operads from $P$ to $Q$; if $f, g$ are two such maps, a morphism from $f$
to $g$ is a natural transformation $\alpha$ from~$f$ to $g$,
that is, the data of an operation $\alpha_c$ in $Q(f(c); g(c))$ for every
colour $c$ of $P$ such that for every operation $p$ in $P(c_1, \dots, c_n;
d)$ we have
\[
  g(p) \circ (\alpha_{c_1}, \dots, \alpha_{c_n}) =
    \alpha_d \circ f(p) \in Q(f(c_1),\ldots,f(c_n);g(d)).
\]
\end{paragr}

\begin{paragr}\label{par:JJJ}
Recall from paragraph \ref{paragr:dend_spines} that we denote by $\I$ the
set
\[ \I = \{i_T\mid T \in \Ob(\Omega)\} \]
of spine inclusions.

Let $J$ be the contractible groupoid on two objects and let $j\colon J \to
\eta$ be the unique map of operads. For any tree $T$, by tensoring $j$ with $T$
we obtain an induced map of operads
\[ j_T\colon J \tensOp T \longto \eta \tensOp T \stackrel{\cong}{\longto} T. \]
Note that $j_\eta$ is canonically isomorphic to $j$.
We will denote by $\J^\flat$ and $\J$ the sets
\[
\J^\flat = \{j_T\mid T \in \Ob(\Omega)\}
\quad\text{and}\quad
\J = \{N_d(j_T)\mid T \in \Ob(\Omega)\}.
\]
The sets $\I$ and $\J$ will be considered as sets of maps in dendroidal sets
or dendroidal spaces depending on the context.
\end{paragr}

\begin{defi}[Cisinski--Moerdijk]
The \emph{model category for complete dendroidal Segal spaces}
or \emph{$\Omega$-spaces} is the left Bousfield localization of the generalized
Reedy model structure on dendroidal spaces by the set $\I \cup \J$.
We will denote this model structure by $\spref{\Omega}_\OmSp$.
\end{defi}

\begin{rem}
This definition is slightly different from the one appearing in
\cite[Definition 6.2]{CisMoerdDendSegal}, which uses the tensor product of
dendroidal sets. The two definitions are immediately seen to be equivalent
using that $N_d(J\tensOp T)\cong N_d(J) \otimes T$ (here~$\otimes$~denotes
the tensor product of dendroidal sets induced by the Boardman--Vogt tensor
product \cite[Section~5]{MoerWeissDend}); see~\cite[Lemma
4.3.3]{WeissThesis}.
\end{rem}

\begin{defi}
The \emph{quasi-category of $\Omega$-spaces} is the localization of the
quasi-category $\qpref{\Omega}$ by the set $\I \cup \J$. We will denote it
by $\OmSp$.
\end{defi}

\begin{prop}
\label{prop:under_dspaces}
The quasi-category underlying the model category of $\Omega$\nbd-spaces is
canonically equivalent to the quasi-category of $\Omega$-spaces.
\end{prop}

\begin{proof}
We indeed have
{\allowdisplaybreaks
\begin{align*}
\U(\spref{\Omega}_{\OmSp})
& =
\U({(\I\cup\J)}^{-1}\spref{\Omega}_{\text{Reedy}}) \\
& \cong
{(\I\cup\J)}^{-1}\U(\spref{\Omega}_{\text{Reedy}}) \\*
& \phantom{=1} \text{(by Proposition~\ref{prop:loc_mcat_qcat})} \\
& \cong
{(\I\cup\J)}^{-1}\U(\spref{\Omega}_\proj) \\*
& \phantom{=1} \text{(since $\U$ only depends on the weak equivalences)}\\
& \cong
{(\I\cup\J)}^{-1} \qpref{\Omega} \\*
& \phantom{=1} \text{(by Proposition~\ref{prop:desc_PA})} \\
& = \OmSp. \qedhere
\end{align*}
}
\end{proof}

\subsection{Autoequivalences of the category of operads}

\begin{lemma}\label{lemma:char_cat}
Let $P$ be an operad. The following are equivalent:
\begin{enumerate}
\item The operad $P$ is a category.
\item\label{item:cat_2} There exists a morphism from $P$ to $\eta$.
\item There exists exactly one morphism from $P$ to $\eta$.
\item\label{item:cat_4} For every non-empty operad $Q$, there exists a
  morphism from $P$ to $Q$.
\end{enumerate}
\end{lemma}

\begin{proof}
The equivalence of the three first properties is obvious.
It is clear that \ref{item:cat_4} implies \ref{item:cat_2}. Conversely, if
$P$ satisfies \ref{item:cat_2} and $Q$ is a non-empty operad, we obtain a
map from $P$ to $Q$ by composing $P \to \eta \to Q$, where the second map
corresponds to the choice of a colour of $Q$.
\end{proof}

\begin{prop}\label{prop:cat_preserved}
Let $F$ be an autoequivalence of $\Oper$. If $C$ is a category, then $F(C)$
is a category. In particular, $F$ restricts to an equivalence of $\Cat$.
\end{prop}

\begin{proof}
Since the empty operad is the initial object of $\Oper$, it has to be
preserved by~$F$. By applying this remark to a quasi-inverse of $F$, we
immediately get that property \ref{item:cat_4} of the previous lemma is
preserved by $F$. The first assertion thus follows from the lemma. This
shows that $F$ induces an endofunctor of $\Cat$. This functor is immediately
seen to be an equivalence of categories by applying the first assertion to a
quasi-inverse of $F$.
\end{proof}

\begin{coro}\label{coro:eta_preserved}
For every autoequivalence $F$ of $\Oper$ we have $F(\eta) \cong \eta$. In
particular, if $P$ is an operad, then there is a natural bijection between the sets
of colours of $P$ and $F(P)$.
\end{coro}

\begin{proof}
Since $\eta$ is the terminal object of $\Cat$, the first assertion follows
from the fact that $F$ restricts to an equivalence of $\Cat$.
The second assertion follows since $\eta$ corepresents the functor ``set of
colours''.
\end{proof}

\begin{defi}
An operad is said to be \emph{discrete} if it can be written as a coproduct
of copies of~$\eta$. An operad is said to be \emph{non-discrete} if it is
not discrete, i.e., if it has at least one non-trivial operation.
\end{defi}

\begin{defi}
A \emph{pseudo-corolla} is an operad~$P$ having a non-trivial operation~$p$
satisfying the following two properties:
\begin{enumerate}
  \item\label{item:pc_i}
  For every non-trivial operation $q$ of $P$ we have $p\sim_\Sigma q$.
  \item\label{item:pc_ii}
  The only colours of~$P$ are the inputs and the output of~$p$.
\end{enumerate}
\end{defi}

\begin{rem}
Roughly speaking, a pseudo-corolla is a corolla where the input and output
colours do not have to be distinct. More precisely, a pseudo-corolla is a
corolla if and only if the input and output colours of its unique
non-trivial operation (up to permutation) are distinct.

Note that a pseudo-corolla is \emph{not} determined by an arity and its
colours. Indeed, there are two pseudo-corollas $P$ with colours $a, b$ and
having an operation~$p$ in~$P(a, a; b)$. One has no other operations and a
trivial action of $\Sigma_2$ and the other one has another operation $q$ in
$P(a, a; b)$ and the transposition of~$\Sigma_2$ acts by exchanging~$p$ and
$q$.

Note also that if $P$ is a pseudo-corolla, then any non-trivial operation of
$P$ satisfies conditions \ref{item:pc_i} and \ref{item:pc_ii} of the
definition.
\end{rem}

\begin{rem}
The purpose of the second condition of the definition is to exclude operads
like $C_n \sqcup \eta$.
\end{rem}

\begin{lemma}
\label{lemma:char_pseudo-cor}
Let $P$ be an operad. Then the following are equivalent:
\begin{enumerate}
\item The operad $P$ is a pseudo-corolla.
\item The operad $P$ is non-discrete and every proper suboperad of $P$ is discrete.
\end{enumerate}
\label{qc_char}
\end{lemma}

\begin{proof}
Every pseudo-corolla satisfies (ii) by definition. Conversely, suppose $P$
satisfies (ii). Since $P$ is non-discrete, it has at least one non-trivial
operation~$p$. If $q$ is another non-trivial operation of $P$, then we must have
$p\sim_\Sigma q$, for otherwise $q$ would generate a non-discrete proper
suboperad of $P$. If $c$ is a colour of $P$, then $c$ is necessarily an
input or the output of~$p$, for otherwise $p$ would generate a non-discrete
proper suboperad of~$P$. This shows that $P$ is a pseudo-corolla.
\end{proof}

\begin{prop}\label{prop:qc_preserved}
Every autoequivalence of $\Oper$ preserves pseudo-corollas.
\end{prop}

\begin{proof}
Let $F$ be an autoequivalence of $\Oper$. Since $F$ preserves $\eta$
(Corollary~\ref{coro:eta_preserved}) and coproducts, it also preserves
discrete operads. Applying the same argument to a quasi-inverse of $F$, we
get that $F$ preserves non-discrete operads. It is then immediate to see
that the characterization of pseudo-corollas given by the previous lemma is
stable under $F$.
\end{proof}

\begin{lemma}\label{lemma:cor_char}
Let $P$ be an operad. Then the following are equivalent:
\begin{enumerate}
\item The operad $P$ is the $n$-corolla.
\item The operad $P$ is a pseudo-corolla satisfying the following property:
  for every pseudo-corolla $Q$, if there exists a map $Q \to P$, then $Q$
  has exactly $n+1$ colours.
\end{enumerate}
\end{lemma}

\begin{proof}
If $P=C_n$ is the $n$-corolla, then it is a pseudo-corolla and if there is a
map~$f$ from a pseudo-corolla $Q$ to $C_n$, then $Q$ has to have exactly
$n+1$ colours, since $f$ sends operations of arity $n$ to operations of arity
$n$ and is surjective on the colours.

Conversely, let $P$ be a pseudo-corolla satisfying (ii). Note that taking
$Q=P$ and the identity map, we get that $P$ has exactly $n+1$ colours. If
$P$ has arity $n$, then $P=C_n$ and we are done. If $P$ had arity $m\ne n$,
then there would be a map from $C_m$ to $P$. But $C_m$ has $m+1$ colours and
$P$ would thus not satisfy (ii).
\end{proof}

\begin{prop}\label{prop:cor_preserved}
Every autoequivalence $F$ of $\Oper$ preserves corollas. In particular, if
$P$ is an operad, then there is a natural bijection between the sets of $n$-ary
operations of $P$ and $F(P)$ for all~$n$.
\end{prop}

\begin{proof}
Using the characterization of $n$-corollas given by the previous lemma, this
immediately follows from the fact that $F$ preserves pseudo-corollas
and the number of colours
(Proposition~\ref{prop:qc_preserved} and Lemma~\ref{lemma:cor_char}).
\end{proof}

\begin{prop}\label{prop:root_preserved}
Let $F$ be an autoequivalence of $\Oper$. Then, for every $n \ge 0$,
$F$~preserves the root map of the corolla $C_n$.
\end{prop}

\begin{proof}
Let $r \colon \eta \to C_n$ be the root map of $C_n$. By
Corollary~\ref{coro:eta_preserved}, $F(\eta)$ is isomorphic to $\eta$ and
hence corepresents colours. Let $r' \colon F(\eta) \to F(C_n)$ be the
corresponding root map of $F(C_n)$. We have to prove that $F(r) = r'$.
Recall moreover that by Proposition~\ref{prop:cor_preserved}, $F(C_n)$ is an
$n$-corolla.

The case $n=0$ being obvious, let us assume $n\geq 1$. Suppose on the
contrary that $F(r) \neq r'$. This means that $F(r) \colon F(\eta) \to
F(C_n)$ is a leaf map (see paragraph~\ref{paragr:omega}). Since $F$ is
faithful, only one colour $\eta \to C_n$ can be sent to the root map. In
particular, when $n = 2$, there exists a leaf map $d \colon \eta \to C_2$
which is not sent to the root map.  Now let $C_2\circ_d C_n$ be the pushout
of the following diagram
\[
\xymatrix{
\eta\ar[r]^r\ar[d]_d & C_n \ar@{.>}[d] \\
C_2\ar@{.>}[r] & C_2\circ_d C_n.
}
\]
Thus, the operad $C_2\circ_d C_n$ is the tree obtained by identifying the
leaf $d$ of $C_2$ with the root of~$C_n$.

Since $F$~preserves pushouts, we obtain a pushout
diagram
\[
\xymatrix{
F(\eta)\ar[r]^-{F(r)}\ar[d]_{F(d)} & F(C_n) \ar[d] \\
F(C_2)\ar[r] & F(C_2\circ_d C_n).
}
\]
A description of this pushout (keeping in mind that $F(r)$ and $F(d)$
are leaf maps) reveals that $F(C_2\circ_d C_n)$ has only non-trivial
operations in arity $2$ and $n$. If $n \ge 2$, this is a contradiction, since
$C_2\circ_d C_n$ has an operation of arity $n + 1$ and $F$ preserves this
property by Proposition~\ref{prop:cor_preserved}. Similarly, if $n=1$, we
get a contradiction since $C_2\circ_d C_1$ has strictly more binary
operations than $F(C_2\circ_d C_1)$.
\end{proof}

\begin{paragr}\label{paragr:can_decomp_iso}
Let $F$ be an autoequivalence of $\Oper$ and let $T$ be a tree. Recall from
paragraph~\ref{paragr:can_decomp} that we have a diagram $\DT_T\colon \Cor/T
\to \Omega$. We will now consider it as a diagram in $\Oper$ using the
inclusion functor $\Omega \hookrightarrow \Oper$. The purpose of
this paragraph is to construct a canonical natural isomorphism $\phi\colon
\DT_T \to F\DT_T$.

Let $(C, C \to T)$ be an object of $\Cor/T$. Suppose first $C = \eta$. By
Corollary~\ref{coro:eta_preserved}, we have $F(\eta) \simeq \eta$ and since
$\eta$ has no non-trivial endomorphism, there is a unique morphism
$\eta \to F(\eta)$. We define $\phi_{(\eta, \eta \to T)}$ to be this unique
morphism $\eta \to F(\eta)$.

Suppose now $T = C_n$ for $n \ge 0$. By
Proposition~\ref{prop:cor_preserved},
$F(C_n)$ is an $n$\nbd-corolla. Since $F$ is fully faithful, it induces a
bijection $\Op(\eta, C_n) \to \Op(F(\eta), F(C_n))$ or, in other words, a
bijection between the colours of $C_n$ and the colours $F(C_n)$. By
Proposition~\ref{prop:root_preserved}, this bijection sends the root of $C_n$ to the root of
$F(C_n)$ and we hence get a bijection between the leaves of the $n$-corolla
$C_n$ and the leaves of the $n$-corolla $F(C_n)$. This bijection determines
a unique morphism $C_n \to F(C_n)$ and we define $\phi_{(C_n, C_n \to T)}$
to be this morphism.

By definition, $\phi_{(C_n, C_n \to T)}$ is the unique morphism from $C_n$
to $F(C_n)$ such that the square
\[
\xymatrix{
\eta\ar[r]^-c\ar[d]_{\phi_{(\eta,\eta\to T)}}\ar[r]^-c&C_n\ar[d]^-{\phi_{(C_n,C_n\to T)}}\\
F(\eta)\ar[r]_-{F(c)}&F(C_n)
}
\]
commutes for every colour $c \colon \eta \to C_n$. This precisely means that
$\phi$ is indeed a natural transformation and hence a natural isomorphism
since its components are obviously isomorphisms.
\end{paragr}

\begin{rem}
The natural isomorphism $\phi$ is actually the unique natural transformation
from $\D_T$ to $F\D_T$, as the components of $\phi$ are determined by the
naturality squares.
\end{rem}

\begin{prop}\label{prop:auto_op_restr}
If $F$ is an autoequivalence of $\Oper$, then $F(T) \cong T$ for every
object~$T\!$ of~$\,\Omega$. In particular, $F$ induces an autoequivalence of
$\,\Omega$.
\end{prop}

\begin{proof}
Using the canonical decomposition of $T$
(Proposition~\ref{prop:can_decomp}), the canonical isomorphism $\DT_T \cong
F\DT_T$ of the previous paragraph and the fact that $F$ commutes with colimits,
we obtain a chain of canonical isomorphisms
\[
T\cong \colim \DT_T \cong \colim F\DT_T \cong F(\colim \DT_T) \cong F(T),
\]
thereby proving the result.
\end{proof}

\begin{coro}\label{cor:ff_auto_op_tree}
The dense inclusion $\Omega\hookrightarrow\Oper$ induces a fully faithful functor $\Aut(\Oper)
\to \Aut(\Omega)$.
\end{coro}

\begin{proof}
This is immediate from the previous proposition and
Proposition~\ref{prop:auto_dense}.
\end{proof}

We will show in Section \ref{sec:auto_Omega} that the category
$\Aut(\Omega)$ is a contractible groupoid. As a corollary, we will obtain
the following theorem:

\begin{thm*}
The category $\Aut(\Oper)$ is a contractible groupoid.
\end{thm*}

\subsection{Autoequivalences of the category of rigid operads}\label{sec:auto_rigid}

\begin{paragr}\label{paragr:rigid_op}
An operad is \emph{rigid} if its underlying category is rigid, that is, if
every invertible unary operation is the identity of a colour. In other
words, an operad $P$ is rigid if it is $j$-local, where $j \colon J\to\eta$ is the
map of paragraph~\ref{par:JJJ}.

We will denote by $\rOper$ the full subcategory of $\Oper$ whose objects are
the rigid operads. Note that the operads induced by trees are rigid. We
hence get a canonical inclusion functor $\Omega \hookto \rOper$. We also
have a canonical inclusion functor $\rCat \hookto \rOper$.
\end{paragr}

\begin{prop}\label{prop:auto_rOper_restr}
If $F$ is an autoequivalence of $\rOper$, then $F(T) \cong T$ for every
object~$T\!$ of~$\,\Omega$. In particular, $F$ induces an autoequivalence of
$\Omega$.
\end{prop}

\begin{proof}
The strategy of the proof is similar to the one used for $\Oper$ in the
previous subsection. Each of the results of that subsection has an obvious
variant for $\rOper$ obtained by inserting the adjective ``rigid'' at
appropriate places. We leave it to the reader to check that the proofs of
these results adapt trivially using the fact that the empty operad, operads
induced by trees, pseudo-corollas, discrete operads and the operad
$C_2\circ_d C_n$ appearing in the proof of
Proposition~\ref{prop:root_preserved} are rigid.
\end{proof}

\begin{coro}\label{cor:ff_auto_rop_tree}
The dense inclusion $\Omega\hookto\rOper$ induces a fully faithful functor $\Aut(\rOper)
\to \Aut(\Omega)$.
\end{coro}

\begin{proof}
This is immediate from the previous proposition and Proposition~\ref{prop:auto_dense}.
\end{proof}

We will show in Section \ref{sec:auto_Omega} that the category
$\Aut(\Omega)$ is a contractible groupoid. As a corollary, we will obtain
the following theorem:

\begin{thm*}
The category $\Aut(\rOper)$ is a contractible groupoid.
\end{thm*}

\subsection{Autoequivalences of the category of
trees}\label{sec:auto_Omega}

\begin{paragr}
Let $A$ be a small category. Let us denote by $\Sigma_A$ the group $\prod_{a
\in \Ob(A)} \aut_A(a)$. To any element $\sigma = (\sigma_a)$ in $\Sigma_A$, we
associate an endofunctor $F_\sigma$ of $A$ in the following way:
\begin{enumerate}
  \item For any object $a$ in $A$, we set $F_\sigma(a) = a$.
  \item For any morphism $f \colon a \to b$ in $A$, we set
    $F_\sigma(f) = \sigma^{}_b f \sigma^{-1}_a$.
\end{enumerate}
This assignment defines a monoid morphism $\Sigma_A \to \aut(A)$.
In general, this morphism is neither injective nor surjective.
Furthermore, for any $\sigma$ and~$\sigma'$ in~$\Sigma_A$, we have a canonical
natural isomorphism from $F_\sigma$ to $F_{\sigma'}$ whose $a$-th component
is~$\sigma'_a \sigma^{-1}_a$.

Let us denote by $\grSigma{A}$ the contractible groupoid on $\Sigma_A$,
that is, the category whose objects are the elements of~$\Sigma_A$ and with
a unique morphism between any two objects. The groupoid $\grSigma{A}$ is
canonically endowed with the strict monoidal structure given by the group
structure of~$\Sigma_A$. By the previous paragraph, we have a canonical
strict monoidal functor $\grSigma{A}\to\Aut(A)$, where $\Aut(A)$ is
endowed with the strict monoidal structure given by composition.
\end{paragr}

The purpose of this subsection is to show that for $A = \Omega$, the
canonical functor $\grSigma{\Omega} \to \Aut(\Omega)$ is an isomorphism of
strict monoidal categories.

\begin{prop}\label{prop:aut_Omega_id_obj}
Every autoequivalence of $\Omega$ is the identity on objects.
\end{prop}

\begin{proof}
The strategy of the proof is similar to the one used for $\Oper$ and
$\rOper$ in the previous subsections. Note that $\Omega$ is a skeletal
category and it thus suffices to show that $F(T) \cong T$ for every tree
$T$.

Lemma~\ref{lemma:char_cat} can be adapted to $\Omega$ by saying that a tree
$T$ is linear if and only if, for every tree $S$, there exists a least one
map from $S \to T$. This shows that $F$ restricts to an equivalence of the
full subcategory of $\Omega$ whose objects are linear trees (which is
nothing but the category $\Delta$). Since $\eta$ is the
terminal object of this category,  it has to be preserved by $F$. In
particular, for every tree $T$, the objects~$T$ and~$F(T)$ have the same
number of colours. Furthermore, the $n$-corolla can be characterized in
$\Omega$ as the unique tree which only has $\eta$ as a proper subobject and
which has $n + 1$ colours. It has therefore to be preserved by $F$.

The proof of Proposition~\ref{prop:root_preserved} can be adapted to show
that $F$ also preserves root maps of corollas (one has to observe
that a diagram $C_2 \leftarrow \eta \rightarrow C_n$, where the arrows are
leaf maps, does not admit a pushout in $\Omega$). If now $T$ is an arbitrary
tree, then using the canonical decomposition of $T$ (Proposition~\ref{prop:can_decomp})
and the natural isomorphism $\phi\colon\DT_T\to F\DT_T$ of functors
$\Cor/T\to\Omega$ constructed in paragraph~\ref{paragr:can_decomp_iso}, we
get a chain of isomorphisms
\[
T\cong \colim \DT_T \cong \colim F\DT_T \cong F(\colim
\DT_T) \cong F(T),
\]
thereby proving the result.
\end{proof}

\begin{paragr}
Let $F$ be an autoequivalence of $\Omega$ and let $T$ be an object of
$\Omega$. Recall from paragraphs~\ref{paragr:can_decomp}
and~\ref{paragr:can_decomp_iso} that we
have a diagram $\DT_T \colon \Cor/T \to \Omega$ and a canonical natural
isomorphism $\phi \colon \DT_T \to F\DT_T$. We will denote by $\sigma(F)_T$ the
automorphism of $T$ given by the composition of the canonical isomorphisms
\[
T = \colim \DT_T \longto \colim F\DT_T \longto F(\colim \DT_T) = F(T) = T.
\]
Unravelling the definitions, this means that $\sigma(F)_T$ is the unique
endomorphism of~$T$ such that for any object $(C, i\colon C \to T)$ of
$\Cor/T$, we have
\[
\sigma(F)_T\circ i =F(i)\circ \phi_{(C,i\colon C\to T)}.
\]
Since maps of $\Omega$ are uniquely determined by their action on the colours,
$\sigma(F)_T$ is uniquely determined by the following property: for every
morphism $c\colon\eta \to T$ of $\Omega$, we have
\begin{equation}\label{eq:sigma2}
  \sigma(F)_T\circ c=F(c).\tag{$\star$}
\end{equation}
(We are using here the equality $\phi_{(\eta,c\colon\eta\to T)} =\id{\eta}$.)
Note that by definition of $\phi$, if $(C_n, C_n \to T)$ is an object of~$\Cor/T$,
we have $\phi_{(C_n, C_n \to T)} \circ c=F(c)$ and hence
\[ \phi_{(C_n, C_n \to T)} = \sigma(F)_{C_n}. \]
We will denote by $\sigma(F)$ the element of $\Sigma_\Omega$ whose
component at a tree $T$ is given by $\sigma(F)_T$.
\end{paragr}

\goodbreak

\begin{lemma}
  The assignment $F \mapsto \sigma(F)$ satisfies the following properties:
  \begin{enumerate}
    \item For any $\sigma$ in $\Sigma_\Omega$, we
      have $\sigma(F_\sigma) = \sigma$.
    \item For any autoequivalence $F\!$ of $\Omega$, we have $F =
      F_{\sigma(F)}$.
  \end{enumerate}
\end{lemma}

\begin{proof}
Let $\sigma$ be an element of $\Sigma_\Omega$. To prove the first point, it
suffices to check that for any colour $c \colon \eta \to T$ of $\Omega$,
we have $\sigma(F_\sigma)_T\circ c = \sigma_T\circ c$.
But we have
\[
\sigma(F_\sigma)_T\circ c=F_\sigma(c)=\sigma_T\circ c,
\]
where the first equality holds by \eqref{eq:sigma2} and the second one by
definition.

Let us prove the second point. By Proposition~\ref{prop:aut_Omega_id_obj},
$F$ is the identity on objects. The same is true for $F_{\sigma(F)}$ by
definition. Since maps of $\Omega$ are determined by their action on the
colours, any autoequivalence of $\Omega$ which is the identity on
objects is determined by its action on the maps whose source is $\eta$.
We thus have to check that for any colour $c\colon\eta\to T$ of a tree $T$,
we have $F_{\sigma(F)}(c) = F(c)$. This is indeed
the case since
\[
F_{\sigma(F)}(c)=\sigma(F)_T\circ c = F(c),
\]
where the first equality holds by definition and the second one by
\eqref{eq:sigma2}.
\end{proof}

\begin{prop}
The monoid morphism $\Sigma_\Omega \to \aut(\Omega)$ is an isomorphism.
\end{prop}

\begin{proof}
The previous lemma precisely says that the map $F \mapsto \sigma(F)$ is an
inverse to the map of the statement.
\end{proof}

\begin{thm}\label{thm:auto_Omega}
The functor $\grSigma{\Omega} \to \Aut(\Omega)$ is an isomorphism of
categories. In particular, the category $\Aut(\Omega)$ is a contractible
groupoid.
\end{thm}

\begin{proof}
The previous proposition states that this functor is bijective on objects.
To conclude, it suffices to show that there exists a unique isomorphism
between $F_\sigma$ and~$F_{\sigma'}$, where $\sigma$ and $\sigma'$ are two
elements of $\Sigma_\Omega$. The functor of the statement gives a map from $F_\sigma$ to
$F_{\sigma'}$. Let us prove its uniqueness. Let $\gamma \colon F_\sigma \to
F_{\sigma'}$ be a natural transformation. For any tree $T$ and any colour $c
\colon \eta \to T$, the naturality square
\[
\xymatrix{
F_\sigma(\eta) = \eta \ar[r]^{F_\sigma(c)} \ar[d]_{\gamma_\eta = \id{\eta}} &
F_\sigma(T) = T \ar[d]^{\gamma^{}_T} \\
F_{\sigma'}(\eta) = \eta \ar[r]^{F_{\sigma'}(c)} &
F_{\sigma'}(T) = T \\
}
\]
shows that the action of $\gamma^{}_T$ on the colours is uniquely
determined. The morphism~$\gamma^{}_T$ is hence uniquely determined, thereby
proving the result.
\end{proof}

\begin{thm}
The categories $\Aut(\Oper)$ and $\Aut(\rOper)$ are
contractible groupoids.
\end{thm}

\begin{proof}
By Corollaries~\ref{cor:ff_auto_op_tree} and~\ref{cor:ff_auto_rop_tree}, the
categories $\Aut(\Oper)$ and $\Aut(\rOper)$ are both full subcategories
of the category $\Aut(\Omega)$. The result then follows immediately from
the previous theorem, since full subcategories of contractible groupoids are
again contractible groupoids.
\end{proof}

\subsection{Autoequivalences of the quasi-category of $\infty$-operads}

\begin{prop}
Let $P$ be an operad. Then the following are equivalent:
\begin{enumerate}
  \item\label{item:P_rigid} The operad $P$ is rigid.
  \item\label{item:HomT_rigid} For any tree $T$, the operad $\HomiOp(T, P)$
    is rigid.
  \item\label{item:Hom_rigid} For any operad $Q$, the operad $\HomiOp(Q, P)$
    is rigid.
\end{enumerate}
\label{prop:rigid_op_char}
\end{prop}

\begin{proof}
Recall that $\eta$ is the unit of the Boardman--Vogt tensor product, hence
$\HomiOp(\eta, P)$ is canonically isomorphic to $P$ for every operad~$P$.
This shows that \ref{item:HomT_rigid} implies~\ref{item:P_rigid}. Clearly,
\ref{item:Hom_rigid} implies \ref{item:HomT_rigid}. Let us prove that
\ref{item:P_rigid} implies \ref{item:Hom_rigid}. Recall that we have a
concrete description of the underlying category of $\HomiOp(Q, P)$ (see
paragraph~\ref{paragr:BV}). If $f$ and $g$ are two objects of this category,
then an isomorphism between them is given by isomorphisms $\alpha_c$ in $P(f(c),
g(c))$, where $c$ ranges through the colours of $Q$, such that $g(q) \circ
(\alpha_{c_1}, \dots, \alpha_{c_n}) = \alpha_d \circ f(q)$.  Since $P$ is
rigid, these isomorphisms have to be the identity. This implies that $f$ has
to be equal to $g$, and that $\alpha$ is an identity, thereby
proving the result.
\end{proof}

The sets $\I$, $\J^\flat$ and $\J$ appearing in the remainder of the section are
those introduced in paragraph~\ref{par:JJJ}.

\begin{prop}\label{prop:rigid_Jflat}
An operad is $\J^\flat$-local if and only if it is rigid.
\end{prop}

\begin{proof}
Let $P$ be an operad. Denote by $e\colon P \to 1$ the unique map of operads
from~$P$ to the terminal operad. For any tree $T$, we have
\[
j_T \orth e
\quad\Longleftrightarrow\quad
j \tensOp T \orth e
\quad\Longleftrightarrow\quad
j \orth \HomiOp(T, e),
\]
where $f\orth g$ is notation for saying that $f$ has the unique
left lifting property with respect to $g$. Since $1$ is the terminal operad,
so is $\HomiOp(T, 1)$ and $\HomiOp(T, e)$ is the unique map from $\HomiOp(T,
P)$ to the terminal operad. This means that $P$ is $j_T$-local if and only
if $\HomiOp(T, P)$ is $j$-local, that is, if and only if $\HomiOp(T, P)$ is
rigid (see paragraph~\ref{paragr:rigid_op}).  The result then follows from
Proposition~\ref{prop:rigid_op_char}.
\end{proof}

\begin{prop}\label{prop:rigid_op_IJ}
A dendroidal set is $(\I \cup \J)$-local if and only if it is the nerve of a
rigid operad.
\end{prop}

\begin{proof}
Proposition~\ref{prop:dend_Segal} precisely says that a dendroidal set is
$\I$-local if and only if it is the nerve of an operad. The previous
proposition shows that such a dendroidal set is $\J$-local if and only if
the operad of which it is the nerve is rigid, thereby proving the result.
\end{proof}

\begin{thm}
The quasi-category $\Aut(\OmSp)$ is a contractible Kan complex.
\end{thm}

\begin{proof}
We are going to apply Proposition~\ref{prop:lemma_aut} to $A = \Omega$ and
$S = \I \cup \J$. Let us check that the hypotheses are fulfilled. Using the
previous proposition, this amounts to verifying that
\begin{enumerate}
  \item objects of $\Omega$ are rigid operads;
  \item autoequivalences of $\rOper$ restrict to autoequivalences of
    $\Omega$.
\end{enumerate}
The first point is obvious and the second point is
Proposition~\ref{prop:auto_rOper_restr}. We can thus apply the
proposition and we get that $\Aut(\OmSp)$ is a full subcategory of
$\Aut(\Omega)$. But $\Aut(\Omega)$ is a contractible groupoid by
Theorem~\ref{thm:auto_Omega} and the result follows.
\end{proof}

\section{Autoequivalences of the \pdfinftyo-category of non-symmetric
\pdfinfty-operads}
\label{sec:nsoperads}

The purpose of this section is to show that the quasi-category of
autoequivalences of the quasi-category of planar $\Omega$-spaces, which we
use as a model for non-symmetric \oo-operads, is the discrete category
$\Zd$. The combinatorics is similar to the symmetric case, the main
differences being due to the fact that a planar tree has no non-trivial
automorphisms.

\subsection{Preliminaries on non-symmetric operads and planar trees}

\begin{paragr}
We will denote by $\nsOper$ the category of (small) non-symmetric coloured
operads, i.e., operads without an action of the symmetric group. There is a
forgetful functor from $\nsOper$ to the category $\nsColl$ of
non\nbd-symmetric collections, and this functor admits a left adjoint.
\end{paragr}

\begin{paragr}
Recall from paragraph~\ref{paragr:assoc_operad} that to every planar tree
$T$, we can associate a non-symmetric collection $K_p(T)$. We define the
non-symmetric operad $\nsOmega(T)$ as the free non-symmetric operad on
$K_p(T)$.
\end{paragr}

\begin{paragr}
The category of planar trees $\nsOmega$, introduced by Moerdijk in
\cite{MoerdToen}, is defined as follows: the objects of $\nsOmega$ are planar
trees (up to isomorphism) and the set of morphisms in $\nsOmega$ from an
object $S$ to an object $T$ is given by
$$
\nsOmega(S, T)=\nsOper(\nsOmega(S),\nsOmega(T)).
$$
By definition, there is a canonical fully faithful functor
$\nsOmega\hookto\nsOper$ and we will always consider $\nsOmega$ as a full
subcategory of $\nsOper$ using this functor.
\end{paragr}

\begin{rem}
One important difference between $\nsOmega$ and $\Omega$ is that
$\aut_{\nsOmega}(T)$ is trivial for every object $T$ of $\nsOmega$.
\end{rem}

\begin{paragr}\label{paragr:can_decomp_planar}
As in paragraph~\ref{paragr:can_decomp}, for any planar tree $T$, we have a
category $\Cor/T$ and a diagram $\mathcal{D}_T \colon \Cor/T\to\nsOmega$.
A similar proof as the one of Proposition~\ref{prop:can_decomp} shows that
for every object $T\!$ of $\,\nsOmega$, the canonical morphism
\[ \colim \DT_T = \colim_{(C, C \to T) \in \Cor/T}\limits\!\!\!\!\! C \,
\longrightarrow\, T \]
is an isomorphism in $\nsOmega$, and that, moreover, the inclusion functor
$\nsOmega \hookrightarrow \nsOper$ preserves this colimit.
\end{paragr}

\begin{paragr}
The category of \emph{planar dendroidal sets} $\Pr(\nsOmega)$ is the
category of presheaves on~$\nsOmega$. The inclusion $\nsOmega\hookto\nsOper$
induces a \emph{planar dendroidal nerve} functor $N_{{\rm ns}, {\rm d}}\colon
\nsOper\to\Pr(\nsOmega)$.
\end{paragr}

\begin{paragr}\label{paragr:planar_spines}
The \emph{spine} of a planar tree $T$ is the planar dendroidal set
\[
  I_T = \colim_{(C, C \to T) \in \Cor/T}\limits\!\!\!\!\! C,
\]
where the colimit is taken in $\Pr(\nsOmega)$. There is a canonical
monomorphism of planar dendroidal sets $i_T\colon I_T\to T$. We will denote
by $\I$ the set
\[ \I = \{i_T\mid T \in \Ob(\nsOmega)\} \]
of spine inclusions.

Let now $X$ be a planar dendroidal set. For any planar tree $T$, the map
$i_T$ induces a \emph{Segal map}
\[
X(T)\cong \Pr(\nsOmega)(T,X)\longrightarrow \Pr(\nsOmega)(I_T, X)\cong \lim_{(C, C \to T) \in \Cor/T}\limits\!\!\!\!\! X(C).
\]
We will say that $X$ \emph{satisfies the Segal condition} if all the Segal
maps are bijections. This exactly means that $X$ is $\I$-local.
\end{paragr}

\begin{prop}[Cisinski--Moerdijk, Weber]\label{prop:pdend_Segal}
The planar dendroidal nerve functor is fully faithful. Moreover, its essential
image consists of the planar dendroidal sets satisfying the Segal condition.
\end{prop}

\begin{proof}
The proof of~\cite[Corollary 2.6]{CisMoerdDendSegal} can be adapted to the
case of planar dendroidal sets. It also follows from the general machinery
of~\cite{Weber}.
\end{proof}

\begin{rem}
The first assertion of the previous proposition precisely means that the
inclusion functor $\nsOmega \hookto \nsOper$ is dense.
\end{rem}

\subsection{The quasi-category of complete planar dendroidal Segal spaces}

\begin{paragr}
The category of \ndef{planar dendroidal spaces} is the category
$\spref{\nsOmega}$ of simplicial presheaves on $\nsOmega$. We will consider
the category of planar dendroidal sets as a full subcategory of the category of
planar dendroidal spaces (as in paragraph~\ref{paragr:dend_spaces}).
\end{paragr}

\begin{paragr}
In the definition of the Boardman--Vogt tensor product of operads, it is
crucial that the operads under consideration are symmetric. However, the tensor product
still makes sense without the symmetries when at least one of the operads
involved is a category.

More precisely, the category of non-symmetric operads is tensored over the
category of categories.  We will denote by $\otimes_\nsOper$ this tensor.
The tensor $\otimes_{\nsOper}$ is closed and we will denote by $\HominsOp$ the
associated enrichment over categories.
This means that if $C$ is a category, and $P$ and $Q$ are non-symmetric
operads, then there is a canonical bijection
\[
\nsOper(C \otimes_{\nsOper} P, Q)\cong \Cat(C, \HominsOp(P, Q)).
\]
The category $\HominsOp(P, Q)$ can be described in the same way as the
underlying category of the operad of morphisms between operads described in
paragraph~\ref{paragr:BV}.
\end{paragr}

\begin{paragr}\label{paragr:JJJ_planar}
Recall from paragraph~\ref{paragr:planar_spines} that we denote by $\I$
the set
\[ \I = \{i_T \mid T \in \Ob(\nsOmega)\} \]
of spine inclusions.

As in the non-planar case, for every planar tree $T$, we have a canonical
map
\[
j_T\colon J \otimes_{\nsOper} T \longto T
\]
of non-symmetric operads, where $J$ denotes the contractible groupoid on
two objects. We will denote by $\J^\flat$ and $\J$ the sets
\[
\J^\flat=\{j_T\mid T\in\Ob(\nsOmega)\}\quad\mbox{and}\quad
\J=\{N_{{\rm ns}, {\rm d}}(j_T)\mid T\in\Ob(\nsOmega)\}.
\]
The sets $\I$ and $\J$ will be considered as sets of maps in planar
dendroidal sets or planar dendroidal spaces depending on the context.
\end{paragr}

\begin{defi}
The \emph{model category for complete planar dendroidal Segal spaces}
or \emph{$\nsOmega$-spaces} is the left Bousfield localization of the injective
model structure on planar dendroidal spaces by the set $\I \cup \J$.
\end{defi}

\begin{defi}
The \emph{quasi-category of $\nsOmega$-spaces} is the localization of the
quasi-category $\qpref{\nsOmega}$ by the set $\I \cup \J$. We will denote it
by $\nsOmSp$.
\end{defi}

\begin{prop}
The quasi-category underlying the model category of $\nsOmega$\nbd-spaces is
canonically equivalent to the quasi-category of $\nsOmega$-spaces.
\end{prop}

\begin{proof}
The proof is the same as the one of Proposition~\ref{prop:under_dspaces}.
\end{proof}

\subsection{Autoequivalences of the category of non-symmetric operads}

\begin{paragr}
For each $n\ge 1$, we denote by $\mu_n$ the mirror permutation in
$\Sigma_n$, that is, the permutation defined by
$$
\mu_n(i)=n-i+1,\quad 1\le i\le n.
$$
We define an endofunctor $\mirror$ of the category $\nsOper$ in the following way:
\begin{enumerate}
\item Given an operad $P$ in $\nsOper$, the operad $\mirror(P)$ has the same
  colours as $P$ and, for colours $c_1, \dots, c_n$ and $c$ of $M(P)$,
  we set
  \[ \mirror(P)(c_1,\ldots, c_n; c) = P(c_{\mu_n(1)}, \ldots, c_{\mu_n(n)};
  c). \]
\item For a map of operads $f\colon P\to Q$ in $\nsOper$,
 the map $M(f)$ is defined on
  components by
  \[ \mirror(f)_{(c_1,\ldots, c_n;c)}(p)=f_{(c_{\mu_n(1)}, \ldots,
  c_{\mu_n(n)}; c)}(p), \]
  where $p$ is an operation in $\mirror(P)(c_1, \ldots, c_n; c)$.
\end{enumerate}
It is easy to check that this functor is indeed well-defined. Obviously,
$\mirror \circ \mirror$ is the identity and $\mirror$ is hence an
autoequivalence. We will call $\mirror$ the \emph{mirror autoequivalence}.
Note that the mirror autoequivalence sends a planar tree to the planar tree
obtained by reversing the orientation of the plane.

The autoequivalence $\mirror$ defines a monoid morphism
\[ \Zd \longto \aut(\nsOper) \]
and hence a strict monoidal functor
\[ \catZd \longto \Aut(\nsOper), \]
where $\catZd$ denotes the discrete category on the set $\Zd$ endowed
with the strict monoidal structure given by the group law of $\Zd$, and
the category $\Aut(\nsOper)$ is endowed with the strict monoidal structure given by
composition of functors.
\end{paragr}

\begin{prop}\label{prop:planar_preserved}
Let $F$ be an autoequivalence of $\nsOper$. Then $F$ preserves $\eta$,
the corollas and the root maps of the corollas.
\end{prop}

\begin{proof}
The proofs of Corollary~\ref{coro:eta_preserved},
Proposition~\ref{prop:cor_preserved} and
Proposition~\ref{prop:root_preserved} are
easily adapted using the following notion of non-symmetric
pseudo-corollas: a \emph{non-symmetric pseudo-corolla} is a non-symmetric
operad $P$ with a unique non-trivial operation $p$ such that the only
colours of $P$ are the inputs and the output of $p$. Non-symmetric
pseudo-corollas can be characterized as in
Lemma~\ref{lemma:char_pseudo-cor}, and non-symmetric $n$-corollas can be
characterized in terms of non-symmetric pseudo-corollas as in
Lemma~\ref{lemma:cor_char}.
\end{proof}

\begin{paragr}
We will denote by $\Sigma$ the group $\prod_{n \ge 0} \Sigma_n$. (Note that
$\Sigma_0$ and $\Sigma_1$ are trivial.) We will associate to every
autoequivalence $F$ of $\nsOper$ an element $\sigma(F)$ in $\Sigma$. Let us
define its components $\sigma(F)_n$.

Let thus $F$ be an autoequivalence of $\nsOper$ and let $n \ge 0$. Since $F$
is fully faithful, it induces a bijection $\nsOper(\eta, C_n) \to \nsOper(F(\eta),
F(C_n))$. By the previous proposition, we have $F(\eta) \cong \eta$ and
$F(C_n) \cong C_n$. Moreover, since $\eta$ and $C_n$ have no non-trivial
automorphisms, these isomorphisms are canonical. We thus obtain an
automorphism of the set $\nsOper(\eta, C_n)$ and, since by the previous
proposition the root map of $C_n$ is preserved, an automorphism of the
input colours of $C_n$, that is, an element in $\Sigma_n$. This permutation
is by definition the component $\sigma(F)_n$ of $\sigma(F)$.

Clearly, if $F = M$ is the mirror autoequivalence, we have $\sigma(M)_n =
\mu_n$. In particular, $\sigma(M)_2$ is the transposition $\tau$ of
$\Sigma_2$.
\end{paragr}

\begin{paragr}\label{paragr:F_twidle}
Given any autoequivalence $F$ of $\nsOper$, we will denote by
$\widetilde{F}$ the autoequivalence
\[
\widetilde{F}=
\begin{cases}
  F & \text{if $\sigma(F)_2=1$}, \\
  M \circ F & \text{if $\sigma(F)_2=\tau$}.
\end{cases}
\]
Note that $\sigma(\widetilde{F})_2 = 1$ for every autoequivalence $F$.
\end{paragr}

\begin{paragr}\label{paragr:trees_Bn}
For every $n\ge 2$, we will denote by $B_n$ the planar binary tree
with $n$ leaves of the following shape:
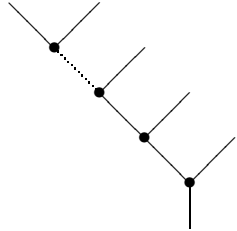
\[
\xy<0.06cm, 0cm>: (0,-15)*{}="1"; (0,-5)*=0{\bullet}="2";
(-10,5)*=0{\bullet}="3"; (10,5)*{}="5";
(-20,15)*=0{\bullet}="6"; (0,15)*{}="7";
(-30,25)*=0{\bullet}="8";(-10,25)*{}="9";
(-40,35)*{}="10";(-20,35)*{}="11";
"1";"2" **\dir{-}; "2";"3" **\dir{-}; "2";"5" **\dir{-}; "3";"6" **\dir{-}; "3";"7" **\dir{-};
"6";"8" **\dir{.};"6";"9" **\dir{-};"8";"10" **\dir{-};"8";"11" **\dir{-};
\endxy
\]
In particular, $B_2$ is the $2$-corolla $C_2$.
\end{paragr}

\begin{lemma}
\label{binary_preserved}
Let $F$ be an autoequivalence of $\nsOper$. Then $\widetilde{F}(B_n)\cong
B_n$ and $\widetilde{F}$~preserves all maps from $\eta$ to $B_n$ for every
$n\ge 2$.
\end{lemma}

\begin{proof}
We prove it by induction on $n \ge 2$. If $n = 2$, then $B_2$ is the
$2$-corolla $C_2$ which is preserved by
Proposition~\ref{prop:planar_preserved}.  Moreover, since $\sigma(\wt{F})_2
= 1$, we know that $\smash{\widetilde{F}}$ preserves all maps from $\eta$ to
$C_2 = B_2$.

Suppose now that the result is true for some $n \ge 2$. We have the
following pushout square
\[
\xymatrix{
\eta \ar[r] \ar[d] & B_n \ar[d]^u \\
B_2\ar[r]_-l & B_{n+1},
}
\]
where the top map is the root map, the  map on the left corresponds to the
left leaf of $B_2$, $u$ is the inclusion of~$B_n$ into the upper part
of~$B_{n+1}$ and $l$ is the inclusion of~$B_2$ into the lower part of
$B_{n+1}$. By induction hypothesis, the maps from $\eta$ to $B_2$ and $B_n$
are preserved by $\wt{F}$. Since $\widetilde{F}$ preserves pushouts, it
has to preserve $B_{n+1}$. Moreover, since $B_{n+1}$ has no non-trivial
automorphisms, it also has to preserve $l$ and $u$.

Let now $f\colon \eta\to B_{n+1}$ be any morphism. Such a map corresponds to
a colour of~$B_{n+1}$ which has to belong either to $B_2$ or to $B_n$. More
precisely, either there exists a map $g\colon \eta\to B_2$ such that
$f=lg$, or there exist a map $h\colon \eta\to B_n$ such that $f=uh$.  Since
$l$ and $u$ are preserved and any map from $\eta$ to $B_2$ or $B_n$ is
preserved by induction hypothesis, we obtain that $\widetilde{F}$
preserves~$f$.
\end{proof}

\begin{prop}
\label{prop:maps_cor_preserved}
If $F$ is an autoequivalence of $\nsOper$, then $\sigma(\wt{F}) = 1$. In
other words, $\widetilde{F}$ preserves all maps from $\eta$ to the $n$-corolla
for every $n \ge 0$.
\end{prop}

\begin{proof}
The cases $n = 0$ and $n = 1$ are trivial, and the case $n = 2$ is true by
definition of $\widetilde{F}$. Let $n\ge 3$ and consider a map $f\colon
\eta\to C_n$. We have a unique map $t \colon C_n \to B_n$ from $C_n$ to $B_n$.
This map corresponds to the total composition of $B_n$. Set $g = tf$
and consider the commutative triangle
\[
\xymatrix{
\eta\ar[r]^-{g}\ar[d]_f & B_n \\
C_n\ar[ur]_t \pbox{.}
}
\]
Since $t$ is injective on colours, the map $f$ is the only map making this
triangle commute. But the map $g$ and $t$ are both  preserved by $\wt{F}$
(the first one by Lemma~\ref{binary_preserved} and the second one by
uniqueness). It follows that $f$ is also preserved by $\widetilde{F}$.
\end{proof}

\begin{prop}\label{prop:auto_nsOper_restr}
If $F$ is an autoequivalence of $\nsOper$, then $\widetilde{F}(T)\cong T$
for every object $T$ of $\nsOmega$. In particular, $F$ sends planar trees to
planar trees and thus induces an autoequivalence of $\nsOmega$.
\end{prop}

\begin{proof}
Recall from paragraph~\ref{paragr:can_decomp_planar} that every planar tree
$T$ is the colimit of a diagram $\DT_T \colon \Cor/T \to \nsOmega \hookto
\nsOper$. By
Propositions~\ref{prop:planar_preserved} and~\ref{prop:maps_cor_preserved},
$\wt{F}$ preserves $\eta$, $C_n$ and all the maps from $\eta \to C_n$.
Moreover, it preserves them up to a canonical isomorphism since these
objects do not have non-trivial automorphisms. We thus have
a canonical isomorphism $\DT_T \cong \wt{F}\DT_T$ and hence a chain of
canonical isomorphisms
\[
T\cong \colim \DT_T \cong \colim \wt{F}\DT_T \cong \wt{F}(\colim \DT_T)
\cong \wt{F}(T),
\]
thereby proving the first assertion.

The second assertion follows from the fact that the mirror autoequivalence
sends planar trees to planar trees.
\end{proof}

\begin{coro}\label{cor:ff_auto_nsop_nstree}
The dense inclusion $\nsOmega\hookrightarrow\nsOper$ induces a fully faithful functor $\Aut(\nsOper)
\to \Aut(\nsOmega)$.
\end{coro}

\begin{proof}
This is immediate from the previous proposition and
Proposition~\ref{prop:auto_dense}.
\end{proof}

We will show in Section~\ref{section:autoequiv_nsoper} that the monoidal
category $\Aut(\nsOmega)$ is isomorphic to the discrete monoidal category
$\catZd$.  As a corollary, we will obtain the following theorem:

\begin{thm*}
The functor $\catZd \to \Aut(\nsOper)$ is an equivalence of monoidal
categories.
\end{thm*}

\subsection{Autoequivalences of the category of rigid non-symmetric operads}

\begin{paragr}
A non-symmetric operad is \emph{rigid} if its underlying category is rigid.
As in the symmetric case, an operad is rigid if and only if it is $j$-local,
where $j \colon J \to \eta$ is the map $j_\eta$ of paragraph~\ref{paragr:JJJ_planar}.

We will denote by $\nsrOper$ the full subcategory of $\nsOper$ whose objects
are the rigid non-symmetric operads.
Note that the non-symmetric operads induced by planar trees are rigid.
Hence, as in the symmetric case, we get canonical inclusion functors
$\nsOmega \hookto \nsrOper$ and $\rCat \hookto \nsrOper$.
\end{paragr}

\begin{paragr}
The mirror autoequivalence obviously preserves rigid non-symmetric operads. It thus
induces a \ndef{mirror autoequivalence} of $\nsrOper$, which in turn induces
a strict monoidal functor
\[ \catZd \longto \Aut(\nsrOper), \]
as in the non-rigid case.

In particular, if $F$ is an autoequivalence of $\nsrOper$, we can define an
autoequivalence $\wt{F}$ of $\nsrOper$ as in
paragraph~\ref{paragr:F_twidle}.
\end{paragr}

\begin{prop}\label{prop:auto_nsrOper_restr}
If $F$ is an autoequivalence of $\nsrOper$, then $\widetilde{F}(T)\cong T$
for every object $T$ of $\nsOmega$. In particular, $F$ sends planar trees to
planar trees and thus induces an autoequivalence of $\nsOmega$.
\end{prop}

\begin{proof}
The strategy of the proof is basically the same as in the case of $\nsOper$
but taking into account the pertinent modifications for rigid non-symmetric
operads as explained in the proof of Proposition~\ref{prop:auto_rOper_restr}
for the case of rigid operads.
\end{proof}

\begin{coro}\label{cor:ff_auto_nsrop_tree}
The dense inclusion $\nsOmega\hookto\nsrOper$ induces a fully faithful functor $\Aut(\nsrOper)
\to \Aut(\nsOmega)$.
\end{coro}

\begin{proof}
This is immediate from the previous proposition and Proposition~\ref{prop:auto_dense}.
\end{proof}

We will show in Section~\ref{section:autoequiv_nsoper} that the monoidal
category $\Aut(\nsOmega)$ is isomorphic to the discrete monoidal category
$\catZd$. As a corollary, we will obtain the following theorem:

\begin{thm*}
The functor $\catZd \to \Aut(\nsrOper)$ is an equivalence of monoidal
categories.
\end{thm*}

\subsection{Autoequivalences of the category of planar trees}

\begin{paragr}
As already observed, the mirror autoequivalence sends planar trees to planar
trees. It thus induces a \ndef{mirror autoequivalence} of $\nsOmega$ that we
will still denote by~$M$, which in turn induces
a strict monoidal functor
\[ \catZd \longto \Aut(\nsOmega), \]
as in the previous two subsections.

In particular, if $F$ is an autoequivalence of $\nsOmega$, we can define an
autoequivalence~$\wt{F}$ of $\nsOmega$ as in
paragraph~\ref{paragr:F_twidle}.
\end{paragr}

\begin{prop}
If $F$ is an autoequivalence of $\nsOmega$, then $\widetilde{F}$ is the
identity on objects.
\end{prop}

\begin{proof}
The proof of Proposition~\ref{prop:aut_Omega_id_obj} can be easily adapted
using the canonical isomorphism $\DT_T \cong \wt{F}\DT_T$ appearing in the
proof of Proposition \ref{prop:auto_nsOper_restr}.
\end{proof}

\begin{prop}\label{prop:auto_nsOmega}
The monoid morphism $\Zd \to \aut(\nsOmega)$ is an isomorphism.
\end{prop}

\begin{proof}
Let $F$ be an autoequivalence of $\nsOmega$. It clearly suffices to show
that the autoequivalence $\wt{F}$ is the identity. By the previous
proposition, we know it is the identity on objects. To prove it is the
identity on morphisms, we can reduce, as in the symmetric case, to the case
of maps from $\eta$ to corollas and the result thus follows from
Proposition~\ref{prop:maps_cor_preserved}.
\end{proof}

\begin{thm}\label{thm:auto_nsOmega}
The functor $\catZd\to \Aut(\nsOmega)$ is an isomorphism of mon\-oidal categories.
\end{thm}

\begin{proof}
The previous proposition states that this functor is bijective on objects.
To conclude, it suffices to show that $\Aut(\nsOmega)$ is a discrete
category. Let $F$ and $G$ be two autoequivalences of $\nsOmega$ and let
$\gamma\colon F\to G$ be a natural transformation. For every planar tree
$T$, we have a morphism $\gamma_T\colon F(T)\to G(T)$. If $T$ is the planar tree
$B_3$ described in paragraph~\ref{paragr:trees_Bn}, then there are no morphisms
$T \to M(T)$ or $M(T)\to T$. This implies that $F$ and $G$ have to be both
equal to the identity autoequivalence or to the mirror autoequivalence.

So let $F$ be either the identity autoequivalence or $M$, and let
$\gamma\colon F\to F$ be a natural transformation. Let $T$ be a planar tree
and $c\colon \eta\to T$ any colour. Then, by naturality, we have a
commutative square
\[
\xymatrix{
F(\eta)=\eta\ar[r]^-{F(c)}\ar[d]_{\gamma_{\eta}=1_{\eta}} & F(T) \ar[d]^{\gamma_T} \\
F(\eta)=\eta\ar[r]_-{F(c)} & F(T).
}
\]
This shows that $\gamma_T$ is the identity on colours and hence the
identity map, thereby proving that $\gamma$ is the identity natural
transformation.
\end{proof}

\begin{thm}
The monoidal categories $\Aut(\nsOper)$ and $\Aut(\nsrOper)$ are equivalent to the discrete
monoidal category $\catZd$.
\end{thm}

\begin{proof}
By Corollaries~\ref{cor:ff_auto_nsop_nstree}
and~\ref{cor:ff_auto_nsrop_tree}, the categories $\Aut(\nsOper)$ and
$\Aut(\nsrOper)$ are both full (monoidal) subcategories of the category
$\Aut(\nsOmega)$. To conclude, it thus suffices to show that every
autoequivalence of $\nsOmega$ lifts to autoequivalences of $\nsOper$ and
$\nsrOper$. This is obvious since, by Proposition~\ref{prop:auto_nsOmega},
the autoequivalences of $\nsOmega$ are the identity and the mirror autoequivalence.
\end{proof}

\subsection{Autoequivalences of the quasi-category of non-symmetric
$\infty$-operads}\label{section:autoequiv_nsoper}

\begin{paragr}\label{paragr:mirror_nsOmSp}
%\ifpdf\else\vskip-\baselineskip\fi
The autoequivalence $M$ of $\nsOmega$ extends formally to an autoequivalence
of the quasi-category $\qpref{\nsOmega}$. It is easy to see that the sets $\I$
and $\J$ of paragraph~\ref{paragr:JJJ_planar} are stable under this
autoequivalence and we thus get an induced autoequivalence~$M$ of
$\nsOmSp = {(\I \cup \J)}^{-1}\qpref{\nsOmega}$.
\end{paragr}

The sets $\I$, $\J^\flat$ and $\J$ appearing in the remainder of the section
are those introduced in paragraph~\ref{paragr:JJJ_planar}.

\begin{prop}
A non-symmetric operad is $\J^\flat$-local if and only if it is rigid.
\end{prop}

\begin{proof}
The proof is a trivial adaptation of the proof of
Proposition~\ref{prop:rigid_Jflat}
\end{proof}

\begin{prop}
A planar dendroidal set is $(\I\cup\J)$-local if and only if it is the nerve
of a rigid non-symmetric operad.
\end{prop}

\begin{proof}
This follows from the previous proposition and
Proposition~\ref{prop:pdend_Segal} as in the proof
of Proposition~\ref{prop:rigid_op_IJ}.
\end{proof}

\begin{thm}
The quasi-category $\Aut(\nsOmSp)$ is canonically equivalent to the discrete category $\catZd$.
\end{thm}

\begin{proof}
We are going to apply Proposition~\ref{prop:lemma_aut} to $A = \nsOmega$ and
$S = \I \cup \J$. Let us check that the hypotheses are fulfilled. Using the
previous proposition, this amounts to verifying that
\begin{enumerate}
  \item objects of $\nsOmega$ are rigid non-symmetric operads;
  \item autoequivalences of $\nsrOper$ restrict to autoequivalences of
    $\nsOmega$.
\end{enumerate}
The first point is obvious and the second point is
Proposition~\ref{prop:auto_nsrOper_restr}. We can thus apply the
proposition and we get that $\Aut(\nsOmSp)$ is a full subcategory of
$\Aut(\nsOmega)$. But $\Aut(\nsOmega)$ is isomorphic to $\catZd$ by
Theorem~\ref{thm:auto_nsOmega}. To conclude, it thus suffices to show that
every autoequivalence of $\nsOmega$ lifts to an autoequivalence of
$\nsOmSp$. This follows from paragraph~\ref{paragr:mirror_nsOmSp}.
\end{proof}

\begin{rem}
Using the monoidal structures described in Remark~\ref{rem:oomon}, one can
show that $\Aut(\nsOmSp)$ and $\catZd$ are equivalent as mo\-noidal
quasi-categories.
\end{rem}

\bibliographystyle{alpha}
\bibliography{biblio}

\end{document}

%% file: macros.tex
\newcommand\pdfinfty{\texorpdfstring{$\infty$}{oo}}
\newcommand\pdfinftyo{\texorpdfstring{$(\infty, 1)$}{(oo, 1)}}
\newcommand\pdfinftyn{\texorpdfstring{$(\infty, n)$}{(oo, 1)}}

\let\cong\simeq
\newcommand\zbox[1]{\makebox[0pt][l]{#1}}
\newcommand\pbox[1]{\zbox{#1}}
\let\ndef\emph

\DeclareMathOperator\Aut{Aut}
\DeclareMathOperator\Map{Map}

\DeclareMathOperator\aut{aut}
\newcommand\grSigma[1]{{(\Sigma_{#1})}_{\mathrm{contr}}}
\newcommand\Zd{\Z/2\Z}
\newcommand\Zdn{(\Zd)^n}
\newcommand\catdisc[1]{#1_\mathrm{disc}}
\newcommand\catZd{\catdisc{(\Zd)}}
\newcommand\catZdn{\catdisc{\Zdn}}
\DeclareMathOperator\Ob{Ob}
\DeclareMathOperator\Ar{Ar}
\newcommand\Ari[1]{\mathrm{Ar}^{\sim}_{#1}}
\DeclareMathOperator\colim{colim}
\DeclareMathOperator\Fun{Fun}
\DeclareMathOperator\qAut{Aut}

\DeclareMathOperator{\Homi}{\underline{\mathsf{Hom}}}

\let\hookto\hookrightarrow

\newcommand\tabdimk[1][m]{%
{\left(
\begin{matrix}
k_1 && k_2 && \cdots && k_{#1} \cr
& k'_1 && k'_2 & \cdots & k'_{{#1}-1}
\end{matrix}
\right)}}

\newcommand\tabdim{{\tabdimk{}}}

\newcommand{\Tht}[2][\@empty]{%
\ifx\@empty#1
\tau_{#2}^{}%
\else
\tau_{#2, #1}%
\fi
}
\newcommand{\Ths}[2][\@empty]{%
\ifx\@empty#1
\sigma_{#2}^{}%
\else
\sigma_{#2, #1}%
\fi
}

\newcommand\id[1]{{1^{}_{#1}}}

\let\nbd\nobreakdash
\newcommand\oo{$\infty$\nbd}

\newcommand\ztr[1]{\tau_{\le 0}{#1}}
\newcommand\loc[2]{{#2}^{-1}#1}

\newcommand\A{\mathcal{A}}
\newcommand\B{\mathcal{B}}
\newcommand\C{\mathcal{C}}
\newcommand\D{\mathcal{D}}
\newcommand\I{\mathcal{I}}
\newcommand\J{\mathcal{J}}
\newcommand\M{\mathcal{M}}
\newcommand\N{\mathcal{N}}
\newcommand\U{\mathcal{U}}
\newcommand\Z{\mathbf{Z}}

\newcommand\DT{\mathcal{D}}

\newcommand\Oper{{\mathcal{O}{\rm p}}}
\let\Op\Oper
\newcommand\nsOper{{\mathcal{O}{\rm p}_{\mathrm{ns}}}}
\newcommand\rOper{\mathcal{O}{\rm p}_{\mathrm{r}}}
\newcommand\nsrOper{\mathcal{O}{\rm p}_{\mathrm{ns,r}}}
\newcommand\Coll{\mathcal{C}{\rm oll}}
\newcommand\nsColl{\mathcal{C}{\rm oll}_{\mathrm{ns}}}

\newcommand\nsOmega{\Omega_{\mathrm{ns}}}

\newcommand\Cor{\mathbb{C}}

\let\wt\widetilde
\let\eps\varepsilon

\newcommand\Set{\mathcal{S}{\rm et}}
\newcommand\Cat{\mathcal{C}{\rm at}}
\newcommand\sCat{\mathcal{C}{\rm at}_{\Delta}}
\newcommand\rCat{\mathcal{C}{\rm at}_{\text{r}}}
\newcommand\mirror{M}

\newcommand\nGraph{\text{$n$-$\mathcal{G}{\rm raph}$}}
\newcommand\nCat[1]{\text{$#1$-$\mathcal{C}{\rm at}$}_{\mathrm{str}}}
\newcommand\nCatr[1]{\text{$#1$-$\mathcal{C}{\rm at}$}_{\mathrm{str, r}}}
\newcommand\Dn[1]{\mathrm{D}_{#1}}
\newcommand\Sn[1]{\mathrm{S}^{#1}}
\newcommand\tr{\mathrm{tr}}

\newcommand\tensOp{\otimes_\Op}
\newcommand\HomiOp{\Homi_\Op}
\newcommand\HominsOp{\Homi_{\nsOper}}
\newcommand\Dend{{\pref{\Omega}}}

\newcommand\SSet{\mathcal{SS}{\rm et}}
\newcommand\longto{\longrightarrow}

\newcommand\ThnSp{\text{$\Theta_n$-$\mathcal{S}p$}}
\newcommand\OmSp{\text{$\Omega$-$\mathcal{S}p$}}
\newcommand\nsOmSp{\text{$\nsOmega$-$\mathcal{S}p$}}

\newcommand\op{{\mathrm{op}}}

\let\bs\backslash
\let\ndef\emph

\newcommand\orth{\operatorname{\perp}}

\newcommand\cN{N_\Delta}

\newcommand\proj{{\text{proj}}}
\newcommand\inj{{\text{inj}}}
\newcommand\ocf{{\text{o}}}
\newcommand\pref[1]{\operatorname{Pr}(#1)}
\newcommand\qpref[1]{\mathcal{P}(#1)}
\newcommand\spref[1]{\operatorname{sPr}(#1)}
\newcommand\sprefp[1]{\spref{#1}_\proj}
\newcommand\sprefpo[1]{\spref{#1}_\proj^\ocf}

\newcommand\boxp{\operatorname{\square}}

\newcommand{\myrightleftarrows}[1]{\mathrel{\substack{\xrightarrow{#1} \\[-.8ex] \xleftarrow{#1}}}}

%% file: article.bbl
\newcommand{\noopsort}[1]{}
\begin{thebibliography}{HHM15}

\bibitem[Ara10]{AraThesis}
D.~Ara.
\newblock {\em Sur les {$\infty$}-groupoïdes de {G}rothendieck et une variante
  {$\infty$}-catégorique}.
\newblock PhD thesis, Université Paris Diderot -- Paris 7, 2010.
\newblock Supervised by G.~Maltsiniotis.

\bibitem[Ara12]{AraThtld}
D.~Ara.
\newblock The groupoidal analogue {$\widetilde{\Theta}$} to {J}oyal's category
  {$\Theta$} is a test category.
\newblock {\em Appl. Categ. Structures}, 20(6):603--649, 2012.

\bibitem[Ara13]{AraStrWeak}
D.~Ara.
\newblock Strict {$\infty$}-groupoids are {G}rothendieck {$\infty$}-groupoids.
\newblock {\em J. Pure Appl. Algebra}, 217(12):1237--1278, 2013.

\bibitem[Ara14]{AraHQCat}
D.~Ara.
\newblock Higher quasi-categories vs higher {R}ezk spaces.
\newblock {\em J. {$K$}-Theory}, 14(3):701--749, 2014.

\bibitem[Ber02]{BergerNerve}
C.~Berger.
\newblock A cellular nerve for higher categories.
\newblock {\em Adv. Math.}, 169(1):118--175, 2002.

\bibitem[Ber07]{BergnerSimpCat}
J.~E. Bergner.
\newblock A model category structure on the category of simplicial categories.
\newblock {\em Trans. Amer. Math. Soc.}, 359(5):2043--2058, 2007.

\bibitem[Ber10]{BergnerSurvey}
J.~E. Bergner.
\newblock A survey of {$(\infty,1)$}-categories.
\newblock In {\em Towards higher categories}, volume 152 of {\em IMA Vol. Math.
  Appl.}, pages 69--83. Springer-Verlag, 2010.

\bibitem[BK72]{BK}
A.~K. Bousfield and D.~M. Kan.
\newblock {\em Homotopy limits, completions and localizations}.
\newblock Lecture Notes in Mathematics, Vol. 304. Springer-Verlag, 1972.

\bibitem[BM11]{BergMoerdGenReedy}
C.~Berger and I.~Moerdijk.
\newblock On an extension of the notion of {R}eedy category.
\newblock {\em Math. Z.}, 269(3-4):977--1004, 2011.

\bibitem[BSP13]{BarSchPrUnicity}
C.~Barwick and C.~Schommer-Pries.
\newblock On the unicity of the homotopy theory of higher categories.
\newblock \href{http://arxiv.org/abs/1112.0040}{arXiv:1112.0040v4 [math.AT]},
  2013.

\bibitem[BV73]{BoardVogt}
J.~M. Boardman and R.~M. Vogt.
\newblock {\em Homotopy invariant algebraic structures on topological spaces}.
\newblock Lecture Notes in Mathematics, Vol. 347. Springer-Verlag, 1973.

\bibitem[CM11]{CisMoerdDend}
D.-C. Cisinski and I.~Moerdijk.
\newblock Dendroidal sets as models for homotopy operads.
\newblock {\em J. Topol.}, 4(2):257--299, 2011.

\bibitem[CM13a]{CisMoerdDendSegal}
D.-C. Cisinski and I.~Moerdijk.
\newblock Dendroidal {S}egal spaces and {$\infty$}-operads.
\newblock {\em J. Topol.}, 6(3):675--704, 2013.

\bibitem[CM13b]{CisMoerdDendSimpOper}
D.-C. Cisinski and I.~Moerdijk.
\newblock Dendroidal sets and simplicial operads.
\newblock {\em J. Topol.}, 6(3):705--756, 2013.

\bibitem[CP86]{cordierporter_vogt}
J.-M. Cordier and T.~Porter.
\newblock Vogt's theorem on categories of homotopy coherent diagrams.
\newblock {\em Math. Proc. Cambridge Philos. Soc.}, 100(1):65--90, 1986.

\bibitem[DS11]{DuggerSpivakMap}
D.~Dugger and D.~I. Spivak.
\newblock Mapping spaces in quasi-categories.
\newblock {\em Algebr. Geom. Topol.}, 11(1):263--325, 2011.

\bibitem[Dug01]{dugger:universal}
D.~Dugger.
\newblock Combinatorial model categories have presentations.
\newblock {\em Adv. Math.}, 164(1):177--201, 2001.

\bibitem[EM06]{ElmMand}
A.~D. Elmendorf and M.~A. Mandell.
\newblock Rings, modules, and algebras in infinite loop space theory.
\newblock {\em Adv. Math.}, 205(1):163--228, 2006.

\bibitem[Gro15]{Groth}
M.~Groth.
\newblock A short course on {$\infty$}-categories.
\newblock \href{http://arxiv.org/abs/1007.2925}{arXiv:1007.2925 [math.AT]},
  2015.

\bibitem[Hel88]{Heller}
A.~Heller.
\newblock Homotopy theories.
\newblock {\em Mem. Amer. Math. Soc.}, 71(383):vi+78, 1988.

\bibitem[HHM15]{HeutsHinichMoerdijk}
G.~Heuts, V.~Hinich, and I.~Moerdijk.
\newblock The equivalence between {L}urie's model and the dendroidal model for
  infinity-operads.
\newblock \href{http://arxiv.org/abs/1305.3658}{arXiv:1305.3658 [math.AT]},
  2015.

\bibitem[Hir03]{hirschhorn:loc}
P.~S. Hirschhorn.
\newblock {\em Model categories and their localizations}, volume~99 of {\em
  Mathematical Surveys and Monographs}.
\newblock American Mathematical Society, 2003.

\bibitem[HS01]{SimpsonHirsch}
A.~Hirschowitz and C.~Simpson.
\newblock Descente pour les {$n$}-champs.
\newblock \href{http://arxiv.org/abs/math/9807049}{arXiv:math/9807049v3
  [math.AG]}, 2001.

\bibitem[Joy97]{JoyalTheta}
A.~Joyal.
\newblock Disks, duality and {$\Theta$}-categories.
\newblock Preprint, 1997.

\bibitem[Joy02]{JoyalQCatKan}
A.~Joyal.
\newblock Quasi-categories and {K}an complexes.
\newblock {\em J. Pure Appl. Algebra}, 175(1-3):207--222, 2002.

\bibitem[Joy08a]{JoyalNotes}
A.~Joyal.
\newblock Notes on quasi-categories.
\newblock Preprint, 2008.

\bibitem[Joy08b]{JoyalQCatAppl}
A.~Joyal.
\newblock The theory of quasi-categories and its applications.
\newblock Lectures at the CRM (Barcelona). Preprint, 2008.

\bibitem[Lur07]{lurie:DAGII}
J.~Lurie.
\newblock {D}erived~{A}lgebraic~{G}eometry~{II}: {N}oncommutative~{A}lgebra.
\newblock \href{http://arxiv.org/abs/math/0702299}{arXiv:math/0702299v5
  [math.CT]}, 2007.

\bibitem[Lur09a]{LurieHTT}
J.~Lurie.
\newblock {\em Higher topos theory}, volume 170 of {\em Annals of Mathematics
  Studies}.
\newblock Princeton University Press, 2009.

\bibitem[Lur09b]{LurieGI}
J.~Lurie.
\newblock {$(\text{Infinity},2)$}-categories and the {G}oodwillie calculus {I}.
\newblock Preprint, 2009.

\bibitem[Lur12]{LurieDAG}
J.~Lurie.
\newblock {D}erived~{A}lgebraic~{G}eometry {V, VII--XIV}.
\newblock Preprints, 2012.

\bibitem[Lur13]{LurieHA}
J.~Lurie.
\newblock Higher algebra.
\newblock Preprint, 2013.

\bibitem[MT10]{MoerdToen}
I.~Moerdijk and B.~To{\"e}n.
\newblock {\em Simplicial methods for operads and algebraic geometry}.
\newblock Advanced Courses in Mathematics. CRM Barcelona. Birkh\"auser/Springer
  Basel AG, Basel, 2010.
\newblock Edited by Carles Casacuberta and Joachim Kock.

\bibitem[MW07]{MoerWeissDend}
I.~Moerdijk and I.~Weiss.
\newblock Dendroidal sets.
\newblock {\em Algebr. Geom. Topol.}, 7:1441--1470, 2007.

\bibitem[Rez01]{RezkSegSp}
C.~Rezk.
\newblock A model for the homotopy theory of homotopy theory.
\newblock {\em Trans. Amer. Math. Soc.}, 353(3):973--1007, 2001.

\bibitem[Rez10a]{RezkThSp}
C.~Rezk.
\newblock A {C}artesian presentation of weak {$n$}-categories.
\newblock {\em Geom. Topol.}, 14(1):521--571, 2010.

\bibitem[Rez10b]{RezkThSpCorr}
C.~Rezk.
\newblock Correction to ``{A} {C}artesian presentation of weak
  {$n$}-categories''.
\newblock {\em Geom. Topol.}, 14(4):2301--2304, 2010.

\bibitem[Sim01]{Simpson}
C.~Simpson.
\newblock Some properties of the theory of $n$-categories.
\newblock \href{http://arxiv.org/abs/math/0110273}{arXiv:math/0110273v1
  [math.CT]}, 2001.

\bibitem[To{\"e}05]{ToenAxm}
B.~To{\"e}n.
\newblock Vers une axiomatisation de la th\'eorie des cat\'egories
  sup\'erieures.
\newblock {\em $K${\nbd-}Theory}, 34(3):233--263, 2005.

\bibitem[TV04]{HAG2DAG}
B.~To{\"e}n and G.~Vezzosi.
\newblock From {HAG} to {DAG}: derived moduli stacks.
\newblock In {\em Axiomatic, enriched and motivic homotopy theory}, volume 131
  of {\em NATO Sci. Ser. II Math. Phys. Chem.}, pages 173--216. Kluwer Acad.
  Publ., 2004.

\bibitem[TV05]{HAGI}
B.~To{\"e}n and G.~Vezzosi.
\newblock Homotopical algebraic geometry. {I}. {T}opos theory.
\newblock {\em Adv. Math.}, 193(2):257--372, 2005.

\bibitem[TV08]{HAGII}
B.~To{\"e}n and G.~Vezzosi.
\newblock Homotopical algebraic geometry. {II}. {G}eometric stacks and
  applications.
\newblock {\em Mem. Amer. Math. Soc.}, 193(902):x+224, 2008.

\bibitem[Web07]{Weber}
M.~Weber.
\newblock Familial 2-functors and parametric right adjoints.
\newblock {\em Theory Appl. Categ.}, 18:No. 22, 665--732, 2007.

\bibitem[Wei07]{WeissThesis}
I.~Weiss.
\newblock {\em Dendroidal sets}.
\newblock PhD thesis, Utrecht University, 2007.
\newblock Supervised by I.~Moerdijk.

\end{thebibliography}
